\documentclass[reqno,twoside,11pt]{amsart}

\usepackage{amsmath,amsfonts,calrsfs,fullpage,amssymb,color,verbatim,eucal,yfonts,mathrsfs}

\newtheorem{Theorem}{Theorem}[section]
\newtheorem{Definition}[Theorem]{Definition}
\newtheorem{Proposition}[Theorem]{Proposition}
\newtheorem{Lemma}[Theorem]{Lemma}
\newtheorem{Corollary}[Theorem]{Corollary}
\newtheorem{Remark}[Theorem]{Remark}

\newtheorem{Hypothesis}[Theorem]{Hypothesis}
\makeatletter
\@addtoreset{equation}{section}

\makeatother

\def\R{\mathbb R}
\def\N{\mathbb N}

\def\C{\mathbb C}
\def\E{\mathbb E}
\def\P{\mathbb P}

\def\eps{\varepsilon}
\def\ds{\displaystyle}

\newcommand{\Tr}{\operatorname{Tr}}

\newcommand{\one}{1\!\!\!\;\mathrm{l}}

\title[Maximal Sobolev regularity in infinite dimension]{\bf Sobolev regularity for a class of second order elliptic PDE's in infinite dimension}\date{}

\author[G. Da Prato]{Giuseppe Da Prato}
\address{Scuola Normale Superiore\\
Piazza dei Cavalieri, 7\\ 
56126 Pisa, Italy}
\email{g.daprato@sns.it}

\author[A. Lunardi]{Alessandra Lunardi (Corresponding Author)}
\address{
Dipartimento di Matematica\\
Universit\`a di Parma\\
Parco Area delle Scienze, 53/A\\
43124 Parma, Italy}
\email{alessandra.lunardi@unipr.it}



\begin{document}

 \begin{abstract}  
We consider an elliptic Kolmogorov equation $\lambda u - Ku = f$ in a separable Hilbert space $H$. The Kolmogorov operator $K$ is associated to an infinite dimensional convex gradient system: $dX = (AX - DU(X) )dt + dW (t)$, where $A $ is a self-adjoint operator in $H$, and $U$ is  a convex lower semicontinuous function. Under mild assumptions we prove that for $\lambda >0$ and  $f\in L^2(H,\nu)$
the weak solution $u$ belongs to the Sobolev space $W^{2,2}(H,\nu)$, where $\nu$ is the log-concave probability measure of the system. Moreover maximal estimates on the gradient of $u$ are proved. The maximal regularity results are used in the study of perturbed nongradient systems, for which we prove that there exists an invariant measure. The general results are applied to Kolmogorov equations associated to reaction-diffusion and Cahn--Hilliard stochastic PDEÕs. 
 \end{abstract}

 \maketitle
 
\section{Introduction}
 
Let $H$ be an infinite dimensional separable Hilbert space   (norm $\|\cdot\|$, inner product  $\langle \cdot,\cdot\rangle$). We are concerned with the differential equation  
\begin{equation}
  \label{e1.1}
  \lambda u-\frac12\mbox{Tr}\;[D^2u]-\langle Ax-DU(x),Du\rangle=f,
  \end{equation} 
where $A:D(A)\subset H\to H$ is a linear self-adjoint    negative operator, and such that $A^{-1}$ is of trace class, $U:H\to \R \cup \{+\infty\}$ is convex, proper, lowerly bounded, and lower semicontinuous. The data are $\lambda>0$ and $f:H\to\R$, the unknown is $u:H\to \R$. $Du$ and $D^2u$ represent first and second derivatives of $u$ and $\mbox{Tr}\;[D^2u]$ is the trace of $D^2u$.

Equation \eqref{e1.1} is the elliptic  Kolmogorov equation corresponding to the    differential stochastic equation
\begin{equation}
  \label{e1.3}
dX=(AX-DU(X))dt+dW(t), 
  \end{equation} 
\begin{equation}
  \label{datoiniz}
X(0)=x,
\end{equation}
where $W(t),\;t\ge 0$, is an $H$-valued cylindrical Wiener process.  Equation \eqref{e1.3} is a typical example of   gradient system. Under suitable assumptions, it has a unique  invariant measure $\nu(dx) = Z^{-1}e^{-2U(x)} 
\mu(dx)$, where $Z=\int_He^{-2U(y)} \mu(dy)$     and $\mu$ is the Gaussian measure in $H$ with zero mean and covariance  $Q=-\frac12\,A^{-1}$. This is the reason to assume $A^{-1}$ of trace class. $Z$  is just a normalization constant in order to have a probability measure. 
Moreover  system  \eqref{e1.3}  is  reversible; that is,  if the law of $X(0)$ coincides with $\nu$,  the reversed process $Y(t)=X(T-t),\; t\in[0,T]$   fulfills again  \eqref{e1.3}; see, for instance, \cite{HP}. 
In statistical mechanics   $\nu$ is   called a   Gibbs measure. 

The above assumptions do not guarantee  well--posedness of problem  \eqref{e1.3}--\eqref{datoiniz};  however, under suitable additional assumptions,  a   solution in a weak sense may  be constructed,  using the general strategy presented in \cite{Michael} and applied in \cite{DPR}.   But  in this paper we shall concentrate on the solutions of the Kolmogorov equation \eqref{e1.1} only. The precise relation between the weak solution to \eqref{e1.1} and the   solution to \eqref{e1.3}--\eqref{datoiniz} is established in the case of Lipschitz continuous $DU$, and in the example of Section 5. In such cases we prove that the expected formula
$$u = \int_0^{+\infty}e^{-\lambda t}\E (f(X(t, \cdot)))dt$$
holds for every $f\in C_b(H)$. 

Throughout the paper we assume that $U$ belongs to a suitable Sobolev space. Then, 
the measure $\nu$ symmetrizes  the operator
$$
\mathcal K u:=\frac12\mbox{Tr}\;[D^2u]+\langle Ax-DU(x),Du\rangle,
$$
since for good functions $u$, $v$ (for instance,  smooth cylindrical functions) we have 
$$
\int_H\mathcal Ku\,v\,d\nu=-\frac12\int_H\langle Du,Dv\rangle\,d\nu .
$$
Accordingly, we say that  $u\in W^{1,2}(H,\nu)$ is a {\em weak solution} of equation \eqref{e1.1} if
\begin{equation}
 \label{e6}
\lambda\int_H u\,\varphi\,d\nu+\frac12\int_H\langle Du,D\varphi\rangle\,d\nu=\int_H f\,\varphi\,d\mu,\quad\forall\;\varphi\in W^{1,2}(H,\nu).
 \end{equation} 
For every $\lambda >0$, the weak solutions to \eqref{e1.1} when $f$ runs in $L^2(H, \nu)$ are 
precisely the elements of the domain of the self-adjoint  realization $K$  of $\mathcal K$ associated to the quadratic form $(u,v)\mapsto \frac12\int_H\langle Du,D\varphi\rangle\,d\nu$. See \S 3.1 for the definition of $K$. 

Existence and uniqueness of a weak solution to  \eqref{e1.1} have been extensively studied, even in more general situations. We quote   \cite{Albeverio} for the Dirichlet form approach  and \cite{DPR} where it was proved that the restriction of $\mathcal K $ to exponential functions is essentially $m$--dissipative in $L^2(H,\nu)$.  However, in all these papers only    $W^{1,2}$ regularity of solutions  was considered.

Our main concern  is the investigation of the second derivative of the weak solution  and of other maximal regularity results.  In Section 3  we shall prove that the weak solution   $u$ of  equation \eqref{e1.1}  has the following properties:
$$(i)\quad  u\in W^{2,2}(H,\nu), \qquad (ii)\quad \ds\int_H\|(-A)^{1/2}Du\|^2d\nu<\infty, $$
and under further assumptions,
$$
(iii) \quad  \int_H\langle D^2 U Du,Du\rangle\,d\nu<\infty .$$
Regularity of the second derivative of $u$ and sharp estimates for $Du$ are challenging problems for the theory of elliptic equations, even in finite dimensions.  (i) is a ``natural" maximal regularity result for elliptic equations, both in finite and in infinite dimensions, while (ii) 
is typical of the infinite dimensional setting; see, for example,  \cite{S,DPZ3}  for the Ornstein--Uhlenbeck operator, when $U\equiv 0$. (iii) is meaningful in the case that $D^2 U$ is unbounded; otherwise it is contained in (i). It was known only in finite dimensions (\cite{LMP}). 

Properties (i)--(iii)  allow us  to study  some perturbations of  $\mathcal K$ of the type ${\mathcal  K_1}={ \mathcal K} + {\mathcal B}$, where
 $$
{\mathcal B}u (x) = \langle B(x),Du(x)\rangle,
 $$
 and $B:H\to H$ is possibly unbounded.  This is the subject of Section 4. Taking advantage of (i)--(iii), we can solve 
  \begin{equation}
  \label{e1.4}
  \lambda u-Ku -\langle  B ,Du\rangle=f,
  \end{equation}
under reasonable assumptions on $B$, when $\lambda$ is sufficiently large. The perturbed operator inherits some of the properties of $K$. For instance, it generates an analytic semigroup that preserves positivity.  
In some cases  we can solve \eqref{e1.4} for every $\lambda >0$, in a different $L^2$ setting. More precisely, adapting arguments from \cite{DPZ2} that involve positivity preserving and compactness,  we are able to prove  the existence of $\rho\in L^2(H,\nu)$ such that  a suitable realization of $\widetilde{K}_1$ of $\mathcal K_1 $ is $m$-dissipative in $L^2(H,\zeta)$ where $\zeta(dx)=\rho(x)\nu(dx)$. Then, equation \eqref{e1.4} can be solved for any $\lambda>0$ and any $f\in L^2(H,\zeta)$, and we prove that $\zeta$ is an invariant measure for the semigroup generated by $\widetilde{K}_1$ in $L^2(H,\zeta)$.
  
It is worth to note that ${\mathcal  K_1}$ is the Kolmogorov operator corresponding to system
\begin{equation}
  \label{e1.5}
dX=(AX-DU(X)+B(X))dt+dW(t),\quad X(0)=x,
  \end{equation} 
which is not a gradient system in general. 
It may be useful in the study of nonequilibrium problems arising in statistical mechanics; see for instance \cite{Jona}. 
Another possible application of the regularity of the second derivative of  the solution $u$ of  \eqref{e1.4} could be to the pathwise uniqueness of   \eqref{e1.5} (see the recent paper \cite{DPFPR}), through the Veretennikov transform. This will be the object of future investigations.  

In Sections 5 and 6  we  show that  the general theory may be applied to Kolmogorov equations of reaction-diffusion and Cahn--Hilliard stochastic PDE's.

%%%%%%%%%%%%%%%%%%%%%%%%%%%%%%%%%%%%%%%%%%%%%%%%%%%%   
\section{Notations and preliminaries}
%%%%%%%%%%%%%%%%%%%%%%%%%%%%%%%%%%%%%%%%%%%%%%%%%%%%%    

In this section we fix  notation and collect several  preliminary results needed in the sequel.
   Though     essentially known, they are scattered in different papers, so we will give details  for the reader's convenience.  Readers familiar with Sobolev spaces in infinite dimensions may  jump to Section 3. 
   
  \vspace{3mm} 
   
Let $H$ be a separable Hilbert space with inner product $\langle \cdot, \cdot\rangle$ and norm $\|\cdot\|$, endowed with a Gaussian measure $\mu := {\mathcal N}_{\;0, Q}$ on the Borel sets of $H$, where  $Q\in {\mathcal L}(H)$ is a self-adjoint  positive operator with finite trace. We choose once and for all an orthonormal basis $\{ e_k:\;k\in \N\}$ of $H$ such that  $Qe_k = \lambda_k e_k$ for $k\in \N$ and set $x_k = \langle x, e_k\rangle $ for each $x\in H$.  We denote by  $P_n$   the orthogonal projection on the linear span of $e_1, \ldots, e_n$.  
For each $k\in \N \cup \{+\infty\}$ we denote by   ${\mathcal F}{\mathcal C}^k_b(H)$  the set of the cylindrical functions  $\varphi(x) = \phi(x_1, \ldots, x_n)$ for some $n\in \N$, with $\phi\in C^k_b(\R^n)$.

%%%%%%%%%%%%%%%%%%%%%%%%%%%%%%%%%%%%%%%%%%%%%%%%%%%
 
\subsection{Sobolev spaces with respect to $\mu$}

%%%%%%%%%%%%%%%%%%%%%%%%%%%%%%%%%%%%%%%%%%%%%%%%%%%%

For $p>1$ we set as usual $p' : =p/(p-1)$. If a function $\varphi:H\mapsto \R$ is Fr\'echet differentiable at $x\in H$, we denote by $D\varphi (x)$ its gradient at $x$. Moreover, we denote by $D_k \varphi(x) = \langle D\varphi (x), e_k\rangle$ its derivative in the direction of $e_k$, for every $k\in \N$.

For $0\leq \theta \leq 1$ and $p>1$ the Sobolev spaces $W^{1,p}_{\theta}(H,\mu)$  are  the completions of ${\mathcal F}{\mathcal C}^1_b(H)$   in the Sobolev norms
$$\|\varphi \|_{W^{1,p}_{\theta}(H, \mu)}^p : =  \int_H (|\varphi |^p +  \|Q^{\theta}D\varphi\|^p)d\mu = 
\int_H  |\varphi |^p +  \bigg(\sum_{k=1}^\infty (\lambda^{ \theta}_k  D_k\varphi )^2\bigg)^{p/2} d\mu .$$
For $\theta=1/2$ they coincide with the usual Sobolev spaces of the Malliavin Calculus; see, for example,  \cite[Chapter 5]{Boga}; for $\theta =0  $ and $p=2$ they are the spaces considered in \cite{DPZ3}.  Such  completions are identified with subspaces of $L^{p}(H,\mu)$ since the integration by parts formula 
\begin{equation}
\label{e1.7}
 \int_H D_k\varphi\,\psi\,d\mu= -\int_H D_k\psi\,\varphi\,d\mu  +\frac{1}{\lambda_k}\int_H x_k\varphi\,\psi\,d\mu, \quad \varphi, \;\psi \in {\mathcal F}{\mathcal C}^1_b(H),
 \end{equation}
allows us to easily  show that the operators $Q^{\theta}D: {\mathcal F}{\mathcal C}^1_b(H)\mapsto 
L^p(H, \mu; H)$ are closable in $L^{p}(H,\mu)$, and the domains of their closures coincide with $W^{1,p}_{\theta}(H,\mu)$. 

Moreover, since $x\mapsto x_k\in L^s(H, \mu)$ for every $s\ge 1$, \eqref{e1.7} is extended by density  to all $ \varphi \in W^{1,q}_{\theta} (H, \mu)$, $\psi\in W^{1,p}_{\theta} (H, \mu)$ such that $1/p + 1/q <1$. In fact, extending \cite[Lemma 9.2.7]{DPZ3} to the case $p \geq  2 $ it is possible to see that it holds for $1/p + 1/q=1$ too.

The spaces $W^{1,p}_{\theta}(H, \mu ;H)$ are defined in a similar way, replacing ${\mathcal F}{\mathcal C}^1_b(H)$ by linear combinations of functions of the type $\varphi e_k$, with $\varphi\in {\mathcal F}{\mathcal C}^1_b(H)$. 

%%%%%%%%%%%%%%%%%%%%%%%%%%%%%%%%%%%%%%%%%%%%%%%%%%%
 
\subsection{Sobolev spaces with respect to $\nu$}

%%%%%%%%%%%%%%%%%%%%%%%%%%%%%%%%%%%%%%%%%%%%%%%%%%%%

Concerning $U$ we  shall assume the following:

\begin{Hypothesis}
\label{Hyp}
 $U:H\to \R\cup \{+\infty\} $ is convex, lower semicontinuous and bounded from below. Moreover $U\in W^{1,2}_{1/2}(H, \mu)$. 
\end{Hypothesis}

We denote by     $\nu$ the log-concave measure
$\nu(dx)= Z^{-1}e^{-2U(x)} 
\mu(dx)$. 
Since $e^{-2U } $ is bounded, $\nu(H)=1$.

 \begin{Lemma}
 \label{l1.2}
 For every $p\geq 1$, ${\mathcal F}{\mathcal C}^{\infty}_b(H)$ is dense in $L^p(H, \nu)$.
 \end{Lemma}
 \begin{proof}
  Since $H$ is separable, then $C_b(H)$ is dense in $L^p(H, \nu)$. 
 Any $f\in C_b(H)$ may be approached in $L^p(H, \nu)$ by the sequence 
 $f_n(x) : =  f(P_nx) $, by the dominated convergence theorem. In its turn, the cylindrical functions $f_n$ are approached by their (finite dimensional) convolutions with smooth mollifiers, that belong to ${\mathcal F}{\mathcal C}^{\infty}_b(H)$. 
 \end{proof}

We may apply the integration by parts formula \eqref{e1.7} with $\psi$ replaced by $\psi e^{-2U}$, that belongs to $W^{1,2}_{1/2}(H,\mu)$ for $\psi \in {\mathcal F}{\mathcal C}^1_b(H)$. We get, for $\varphi$, $\psi\in   {\mathcal F}{\mathcal C}^1_b(H)$ and $h\in \N$, 
\begin{equation}
\label{e1.9}
\int_H D_h\varphi\,\psi\,d\nu+\int_H D_h\psi\,\varphi\,d\nu=  2 \int_H
D_hU\,\varphi\,\psi  \,d\nu+\frac1{\lambda_h}\int_H x_h \varphi\,\psi\,d\nu  .
\end{equation} 
 Once again, the Sobolev spaces  associated to the measure $\nu$ are introduced in a standard way with the help of the  integration by parts formula \eqref{e1.9}. We recall that  $ {\mathcal L}_2(H)$ is the space of the Hilbert--Schmidt operators, that are the bounded linear operators  $L:H\mapsto H$ such that $\|L\|_{{\mathcal L}_2(H)}^2 :=\sum_{h,k=1}^\infty\langle Le_h, e_k\rangle^2 <\infty$.

\begin{Lemma}
\label{l1.3}
For all   $q\geq 2$ the operators
\begin{equation}
\label{definizione}
D: {\mathcal F}{\mathcal C}^1_b(H)\mapsto 
L^q(H, \nu; H), \quad Q^{\pm 1/2}D: {\mathcal F}{\mathcal C}^1_b(H)\mapsto 
L^q(H, \nu; H),
\end{equation}
\begin{equation}
\label{definizione2}( D,  D^2 ): {\mathcal F}{\mathcal C}^2_b(H) \mapsto L^{q}(H,\nu; H)\times L^{q}(H,\nu; {\mathcal L}_2(H))
\end{equation}
are closable in $L^q(H, \nu)$. 
\end{Lemma} 
\begin{proof}
Let $(\varphi_n)\subset {\mathcal F}{\mathcal C}^1_b(H)$ converge to $0$ in $L^{q}(H,\nu)$ and be such that 
$Q^{\theta}D\varphi_n\to W$ in $L^q(H, \nu; H)$, with  $\theta =0$ or $\theta = 1/2$ or $\theta = -1/2$ . Then for every $h\in \N$ the sequence 
$(\langle Q^{\theta}D\varphi_n, e_h\rangle) = (\lambda_h^{\theta}D_h\varphi_n)$ converges to $ \langle W, e_h\rangle $ in $L^{q}(H,\nu)$. By formula \eqref{e1.9} for each $\psi\in   {\mathcal F}{\mathcal C}^1_b(H)$ we have
\begin{equation}
\label{serve}
\int_H D_h\varphi_n\,\psi\,d\nu+\int_H D_h\psi\,\varphi_n\,d\nu=  2 \int_H
D_hU\,\varphi_n\,\psi  \,d\nu+\frac1{\lambda_k}\int_Hx_h \varphi_n\,\psi\,d\nu ,
\end{equation}
and letting $n\to \infty$, we get 
$$\lim_{n\to \infty} \int_H D_h\varphi_n\,\psi\,d\nu = \lim_{n\to \infty} \int_H \lambda_h^{-\theta} \langle W, e_h\rangle \,\psi\,d\nu =0.$$
Since $  {\mathcal F}{\mathcal C}^1_b(H)$ is dense in $L^{q'}(H,\nu)$, then $ \langle W, e_h\rangle =0$   $\nu$-a.e. for every $h\in \N$, hence $W=0$ $\nu$-a.e., and the first statement is proved. 

The proof of the second statement is similar. If $(\varphi_n)\subset {\mathcal F}{\mathcal C}^2_b(H)$ converge to $0$ in $L^{q}(H,\nu)$ and $ D\varphi_n\to W$ in $L^q(H, \nu; H)$, 
$ D^2\varphi_n   \to {\mathcal Q}$ in $L^{q}(H,\nu; {\mathcal L}_2(H))$, by the first part of the proof we have $W=0$, so that for every $k\in \N$, $D_k\varphi_n\to 0$ in $L^{q}(H,\nu)$. On the other hand, for each $h, k\in \N$, $\langle  D^2\varphi_n  e_h, e_k\rangle =  D_{hk}\varphi_n$ goes to $ \langle  {\mathcal Q}e_h, e_k\rangle$ in 
$L^{q}(H,\nu)$. 
Formula \eqref{e1.9} applied to $D_k\varphi_n$ instead of $\varphi$ reads as
$$\int_H D_{hk}\varphi_n\,\psi\,d\nu+\int_H D_h\psi\,D_k\varphi_n\,d\nu =  2 \int_H
D_hU\,D_k\varphi_n\,\psi  \,d\nu+\frac1{\lambda_k}\int_H x_k D_k \varphi_n\,\psi\,d\nu  , $$
for all $\psi\in   {\mathcal F}{\mathcal C}^1_b(H)$. Letting $n\to \infty$ we get
$$ \lim_{n\to \infty} \int_H D_{hk}\varphi_n\,\psi\,d\nu
=  \lim_{n\to \infty} \int_H 
  \langle  {\mathcal Q}e_h, e_k\rangle \,\psi\,d\nu =0.$$
Then, $ \langle  {\mathcal Q}e_h, e_k\rangle =0$ a.e. for each $h$ and $k$, so that ${\mathcal Q}=0$, $\nu$-a.e. 
\end{proof}

\begin{Remark}
\label{pgenerale}
{\em We remark that the restriction $q\geq 2$ comes from the integral $\int_H
D_hU\,\varphi_n\,\psi  \,d\nu$ in \eqref{serve}, where $D_hU  \in L^{2}(H, \nu)$ as a consequence of  Hypothesis \ref{Hyp}. If $\|DU\|\in L^p(H, \mu)$ for some $p>2$ the proof of Lemma \ref{l1.3} works for any $q\geq p'$. }
\end{Remark}

\begin{Definition}
\label{d1.4}
For  $q\geq 2$ we still denote by $ D$, $Q^{1/2} D$, $Q^{-1/2} D$, and by 
$(D,  D^2 )$ the closures in $L^q(H, \nu)$ of the operators defined in \eqref{definizione}, \eqref{definizione2}. 

We denote by $W^{1,q} (H, \nu)$ and by $W^{1,q}_{1/2}(H, \nu)$, $W^{1,q}_{-1/2}(H, \nu)$,  the domains of $ D$, $Q^{1/2} D$, $Q^{-1/2} D$ in $L^q(H, \nu)$, respectively, and by $W^{2,q}(H, \nu)$ the domain of $( D,  D^2 )$ in $L^q(H, \nu)$. 
\end{Definition}
 
Then, $W^{1,q}(H, \nu)$, $W^{1,q}_{\pm 1/2}(H, \nu)$ and $W^{2,q}(H, \nu)$ are Banach spaces with the norms
$$\| u \|_{W^{1,q} (H, \nu)}^q = \int_H |u|^q d\nu + \int_H \| Du\|^q d\nu , $$
$$\| u \|_{W^{1,q}_{\pm1/2}(H, \nu)}^q = \int_H |u|^q d\nu + \int_H \|Q^{\pm 1/2}Du\|^q d\nu , $$
$$\| u \|_{W^{2,q} (H, \nu)}^q =\| u \|_{W^{1,q} (H, \nu)}^q + \int_H \| D^2u \|_{{\mathcal L}_2(H)}^q d\nu . $$
Denoting by $D_ku :=  \lambda_k^{-\theta} \langle  Q^{\theta}Du, e_k\rangle $, with $\theta \in \{0, 1/2, -1/2\}$, 
$D_{hk}u :=   $ $\langle  D^2u  \,e_h, e_k\rangle $, the above Sobolev norms  may be written in a more 
explicit way as 
$$\| u \|_{W^{1,q} (H, \nu)}^q = \int_H |u|^q d\nu + \int_H \bigg(\sum_{k\in \N} (D_ku)^2\bigg)^{q/2} d\nu , $$
$$\| u \|_{W^{1,q}_{\pm 1/2}(H, \nu)}^q = \int_H |u|^q d\nu + \int_H \bigg(\sum_{k\in \N}\lambda_k^{\pm 1}(D_ku)^2\bigg)^{q/2} d\nu , $$
$$\| u \|_{W^{2,q} (H, \nu)}^q =\| u \|_{W^{1,q} (H, \nu)}^q + \int_H 
 \bigg(\sum_{h,k\in \N} (D_{hk}u)^2\bigg)^{q/2}
d\nu  
= \| u \|_{W^{1,q} (H, \nu)}^q + \int_H \mbox{Tr}\;([D^2u]^2)d\nu. $$
For $q=2$,  such spaces  are Hilbert spaces with the respective scalar products
$$\langle u, v\rangle _{W^{1,2}(H, \nu)}  = \int_H  u\,v\, d\nu + \int_H \sum_{k\in \N}
 D_ku D_kv\, d\nu  , $$
$$\langle u, v\rangle _{W^{1,2}_{\pm 1/2}(H, \nu)}  = \int_H  u\,v\, d\nu + \int_H \sum_{k\in \N}
\lambda_k^{\pm 1}D_ku D_kv\, d\nu  , $$
$$\langle u, v\rangle _{W^{2,2} (H, \nu)}  = \langle u, v\rangle _{W^{1,2} (H, \nu)}+ \int_H \sum_{h,k\in \N} D_{hk}uD_{hk}v\,d\nu . $$

\begin{Remark} 
\label{r1.5}
{\em Let us make some remarks about the above definitions.  }
\end{Remark}
\begin{enumerate}
\item 
 It follows immediately from the definition that for every $u \in W^{1,p} (H, \nu)$ and $\varphi\in C^1_b(\R)$, the superposition $\varphi \circ u$ belongs to $W^{1,p} (H, \nu)$, and $D(\varphi \circ u) = (\varphi' \circ u)Du$. This fact will be used frequently in the sequel. 
\item
 Formula \eqref{e1.9} holds for each $\varphi \in  {\mathcal F}{\mathcal C}^1_b(H)$, 
$\psi \in W^{1,q} (H, \nu)$ with $q\geq 2$. Indeed, it is sufficient to approach $\psi$ by a sequence of cylindrical functions in ${\mathcal F}{\mathcal C}^1_b(H)$, and to use \eqref{e1.9} for the approximating functions, recalling that $D_hU$, $x_h \in L^2(H, \nu)$.  
\item
 Similarly, \eqref{e1.9} holds for 
$\varphi \in W^{1,p} (H, \nu)$, $\psi \in W^{1,q} (H, \nu)$
such that $1/p + 1/q \leq 1/2$. 
\end{enumerate}

\subsubsection{Positive and negative parts of  elements of  $W^{1,2}(H,\nu)$}

The following technical lemma will be used later to study positivity of solutions of \eqref{e1.1}. 
\begin{Lemma}
\label{l1.6}
Let  $u\in W^{1,2}(H, \nu)$. Then $|u|$ $($and consequently, $u^+=\sup\{u,0\}$, $u^-=\sup\{-u,0\}$$)$ belongs to $W^{1,2}(H, \nu)$, and 
$D|u| = {\it sign} \,u \,Du$. Moreover $Du=0$ a.e. in the set  $u^{-1}(0)$, and $Du^+= Du\,\one_{\{u\geq 0\}} = Du\,\one_{\{u > 0\} }$, $Du^-= -Du\,\one_{\{u\leq 0\}} = -Du\,\one_{\{u < 0\} }$. 
\end{Lemma}
\begin{proof}
Set $f_{n}(\xi) = \sqrt{\xi^2 +1/n}$, $\xi\in \R$. If $(u_n)$ is a sequence of functions in ${\mathcal F}{\mathcal C}^1_b(H) $ that approach  $u$ in $W^{1,2}(H, \nu)$ and pointwise a.e., the functions $f_{n}\circ u_n$ belong to ${\mathcal F}{\mathcal C}^1_b(H) $ and approach $| u |$ in $W^{1,2}(H, \nu)$. Indeed, they converge to $| u |$ in $L^2(H, \nu)$ by the dominated convergence theorem, and  $D(f_{n}\circ u_n) = f_{n}'\circ u_n Du_n$ converge to
${\it sign} \, u\, Du$ in $L^2(H, \nu; H)$. The first statement follows.

Let us prove that $Du$ vanishes a.e. in the kernel of $u$.  It is sufficient to prove that for every $u\in W^{1,2}(H, \nu)$ and $i\in \N$  we have
\begin{equation}
\label{e1.10}
\int_{\{u=0\}} D_iu\,\varphi \,d\nu =0, \quad \varphi\in {\mathcal F}{\mathcal C}^1_b(H). 
\end{equation}
Indeed, since  ${\mathcal F}{\mathcal C}^1_b(H)$ is dense  in $L^2(H, \nu)$, \eqref{e1.10} implies that 
$D_iu\,\one_{\{u=0\}}$ is orthogonal to all elements of   $L^2(H, \nu)$, hence it vanishes a.e.

Let  $\theta :\R\mapsto \R$ be a smooth function with support contained in   $[-1,1]$, with values in  $[0,1]$ and such that  $\theta(0)=1$. For $\eps >0$ set $\theta_{\eps}(\xi) = \theta (\xi/\eps)$. The functions $\theta_{\eps}\circ u$ have  values in  $[0,1]$ and converge pointwise to  $\one_{\{u=0\}}$. Moreover, they belong to  $W^{1,2}(H, \nu)$
and we have  $D_i(\theta_{\eps}\circ u) = (\theta_{\eps}'\circ u) D_iu = (\theta'\circ u/\eps)D_iu/\eps$. Integrating 
we obtain
$$\int_H D_iu\,\varphi \,(\theta_{\eps}\circ u)\, d\nu = - \int_H u\, D_i\varphi \,(\theta_{\eps}\circ u)\, d\nu $$
$$- \int_H u\, \varphi \,D_i(\theta_{\eps}\circ u)\, d\nu +2 \int_H u\, \varphi \,(\theta_{\eps}\circ u)\, D_iU\,d\nu+ \frac{1}{\lambda_i}\int_H x_i\, u\,\varphi \,(\theta_{\eps}\circ u)\, d\nu $$
As $\eps \to 0$ we obtain by the dominated convergence theorem
$$\lim_{\eps \to 0}\int_H D_iu\,\varphi \,(\theta_{\eps}\circ u)\, d\nu = \int_{\{u=0\}} D_iu\,\varphi \,d\nu , $$
$$\lim_{\eps \to 0}\int_H u\, D_i\varphi \,(\theta_{\eps}\circ u)\, d\nu =  \int_{\{u=0\}}  u\, D_i\varphi \, d\nu =0,$$
$$ \lim_{\eps \to 0} \int_H u\, \varphi \,(\theta_{\eps}\circ u)\, D_iU\,d\nu =  \int_{\{u=0\}} 
u\, \varphi \, D_iU\,d\nu =0,$$
$$\lim_{\eps \to 0} \frac{1}{\lambda_i} \int_H x_i\, u\,\varphi \,(\theta_{\eps}\circ u)\, d\nu = \frac{1}{\lambda_i} \int_{\{u=0\}} x_i\, u\,\varphi \,  d\nu =0.$$
The integral $\int_H u\, \varphi \,D_i(\theta_{\eps}\circ u)\, d\nu$ vanishes too as $\eps\to 0$, by the dominated convergence theorem. 
Indeed the support of $u\,\varphi \,D_i(\theta_{\eps}\circ u)$ is contained in  $u^{-1}([-\eps, \eps])$ so that its modulus is bounded by   $\|\theta '\|_{\infty}\|\varphi\|_{\infty}$. Moreover it converges to $0$ pointwise as $\eps\to 0$. So, letting   $\eps\to 0$ we obtain  \eqref{e1.10}.

Once we know that $Du$ vanishes a.e. in the kernel of $u$, the formulas for $Du^+$ and $Du^-$ follow from the equalities $u^+ = (|u|+u)/2$, $u^- = (|u|-u)/2$. 
\end{proof}

\subsubsection{Functional inequalities and embeddings}

Under some additional assumptions important functional inequalities hold in the space $W^{1,2}(H, \nu)$. 

\begin{Hypothesis}
\label{Hyp1}
$U\in W^{1,2}_0(H, \mu)$ and  $\|DU\|\in L^p(H, \mu)$ for some $p>2$. 
\end{Hypothesis}

We recall that since $A$ is invertible and $-A^{-1}$ is nonnegative and compact, then
$$-\omega:= \sup \{ \langle Ax, x\rangle:\;x\in D(A)\} <0. $$

\begin{Proposition}
Let Hypotheses \ref{Hyp} and  \ref{Hyp1} hold. Then the following Poincar\'e and Logarithmic Sobolev inequalities hold. 
\begin{equation}
\label{Poinc}
\int_H \bigg(\varphi -\int_H \varphi 
\,d\nu\bigg)^2  d\nu \leq \frac{1}{2\omega} \int_H \|D\varphi\|^2 d\nu  , \quad \varphi \in W^{1,2}(H, \nu), 
\end{equation}
\begin{equation}
\label{LogSob}
\int_H \varphi^2 \log(\varphi^2) d\nu \leq \frac{1}{\omega}  \int_H \|D\varphi\|^2 d\nu +   \int_H \varphi^2 d\nu \log\bigg( \int_H \varphi^2 d\nu\bigg), \quad \varphi \in W^{1,2}(H, \nu).
\end{equation}
\end{Proposition}
\noindent For the proof we refer to \cite[Section 12.3.1]{DPZ3}.

Another useful property is the compact embedding of $W^{1,2}(H, \nu)$ in $L^2(H, \nu)$; see \cite{DPDG}. 

\begin{Proposition}
\label{p1.7}
Under Hypotheses \ref{Hyp} and  \ref{Hyp1},  $W^{1,2}(H, \nu)$ is compactly embedded in $L^2(H, \nu)$. 
\end{Proposition}
\begin{proof}
Let  $(f_n)$ be a bounded sequence  in $W^{1,2}(H, \nu)$. We look for a subsequence that converges in $L^2(H, \nu)$. 
By the Log--Sobolev inequality \eqref{LogSob} the sequence is uniformly integrable, and hence it is sufficient to find a subsequence that converges almost everywhere. 

The sequence $(f_n\,e^{-U})$ is bounded  in $W^{1,q}_0(H, \mu)$, with $q=2p/(2+p) \in (1, 2)$. Indeed, it is bounded in $L^2(H, \mu)$, and  hence it is bounded in $L^q(H, \mu)$, moreover  $D(f_n\,e^{-U}) = D f_n\,e^{-U} - f_n DU\,e^{-U}$. Once again, $ \| D f_n\,e^{-U}\|$ is bounded in $L^2(H, \mu)$, while the second addendum $f_n DU\,e^{-U}$ satisfies
$$\int_H \|f_n DU\,e^{-U}\|^qd\mu \leq \bigg(\int_H f_n^2 e^{-2U}d\mu\bigg)^{q/2}
\bigg(\int_H \|DU\|^{2q/(2-q)}  d\mu\bigg)^{(2-q)/q} $$
$$= \|f_n\|_{L^2(H, \nu)}^q \bigg(\int_H \|DU\|^{p}  d\mu\bigg)^{(2-q)/q}$$
so that it is bounded in $L^q(H, \mu)$. 

Since the embedding  $W^{1,q}_0(H, \mu)\subset L^q(H, \mu)$ is compact \cite{CG}, there exists a subsequence that converges
 in $L^q(H, \mu)$ and a further subsequence that converges pointwise $\mu$-a.e. and also  $\nu$-a.e, since $\nu$ is absolutely continuous with respect to  $\mu$. 
\end{proof}

%%%%%%%%%%%%%%%%%%%%%%%%%%%%%%%%%%%%%%%%%%%%%%
\subsection{Moreau--Yosida approximations}
%%%%%%%%%%%%%%%%%%%%%%%%%%%%%%%%%%%%%%%%%%%%%

An important tool in our analysis are the  Moreau--Yosida approximations of $U$ defined for $\alpha >0$ by 
\begin{equation}
\label{e1.8}
U_\alpha(x)=\inf\left\{U(y)+\frac{|x-y|^2}{2\alpha},\;y\in H\right\},\quad x\in H .
\end{equation}
We recall that $U_{\alpha}(x) \leq U(x)$ and $U_{\alpha}(x)$ converges monotonically  to $U(x)$ for each $x$ as $\alpha \to 0$. Moreover, each $U_{\alpha}$ is differentiable at any point,   $DU_{\alpha}$ is Lipschitz continuous, and
 $\|DU_{\alpha}\|$ converges monotonically  to $\|D_0U\|$,  at any $x$ such that the subdifferential of $U(x)$ is not empty.  Here, $D_0U(x)$ is the element with minimal norm in the subdifferential of $U(x)$. At such points we have
\begin{equation}
\label{e4.5}
 \|DU_{\alpha}(x) - D_0U(x) \|^2 \leq \|D_0U(x)\|^2 -  \|DU_{\alpha}(x) \|^2 ;
 \end{equation}
see, for example,  \cite[Chapter 2]{Brezis}. 
If in addition $U\in C^2$, then $D_0U=DU$, and we have convergence of the second order derivatives, as the next lemma shows. 
  
\begin{Lemma}
\label{l2.9}
Let $U:H\mapsto \R$ be convex and $C^2$. Then  $\lim_{\alpha \to 0} D^2U_{\alpha}(x) = D^2U (x) $ in ${\mathcal L}(H)$ for all $x\in H$. 
\end{Lemma}
\begin{proof}
For each $x\in H$ set $y_{\alpha }(x) = (I + \alpha DU)^{-1}(x)$, so that
\begin{equation}
\label{e2.21}
y_{\alpha }(x) + \alpha DU(y_{\alpha }(x)) = x, 
\end{equation}
and by  \cite[Chapter 2]{Brezis}, 
\begin{equation}
\label{DUalpha}
DU_{\alpha}(x) = DU(y_{\alpha}).
\end{equation}
Since $U$ is convex, then $\langle DU(x) - DU(y_{\alpha }(x)) , \alpha DU(y_{\alpha }(x)) \rangle $
$=$ $  \langle DU(x) - DU(y_{\alpha }(x)) ,x-y_{\alpha }(x)\rangle $ $\geq 0$. Taking the scalar product with $DU(y_{\alpha }(x))$ yields  $\| DU(y_{\alpha }(x))\| \leq \|DU(x)\|/(1-\alpha)$, and letting $\alpha \to 0$ in  \eqref{e2.21} we get 
$$
\lim_{\alpha \to 0} y_{\alpha}(x) = x,\quad\forall\;x\in H.
$$
Now it is clear that $y_{\alpha}$ is of class $C^1$, and 
differentiating \eqref{e2.21} yields
\begin{equation}
\label{e2.22}
y_{\alpha }'(x) + \alpha D^2U(y_{\alpha }(x))y_{\alpha }'(x)  = I.
\end{equation}
Since $U$ is convex, 
$$
\|y_{\alpha }'(x) \|_{{\mathcal L}(H)}\le 1,
$$
so that, letting $\alpha \to 0$ in \eqref{e2.22} and recalling that $D^2U$ is continuous,  we obtain 
$$
\lim_{\alpha\to 0}y_{\alpha }'(x)=I.
$$
On the other hand, differentiating   identity  \eqref{DUalpha} gives
$D^2U_{\alpha}(x) $ $=$ $ D^2U(y_{\alpha }(x))\cdot y_{\alpha }'(x)$ which yields the statement. 
\end{proof}

%%%%%%%%%%%%%%%%%%%%%%%%%%%%%%%%%%%%%%%%%%%%%%%%%%%
\section{Elliptic problems}
\label{ellittica}
%%%%%%%%%%%%%%%%%%%%%%%%%%%%%%%%%%%%%%%%%%%%%%%%%%%

This section is devoted to the main result of the paper. In Section 3.1 we prove existence and uniqueness of  a weak solution $u$ of equation \eqref{e1.1}. Section 3.2 is devoted to the particular case that $DU$  is Lipschitz continuous. This is an intermediate step in order to prove  in Section 3.3 that under Hypothesis  \ref{Hyp}  we have
$$
u\in W^{2,2}(H,\nu)\cap W^{1,2}_{-1/2}(H,\nu).
$$
In Section 3.4 we show that if in addition $U$ is twice continuously differentiable then
$$
\int_H\langle D^2U(x)Du(x),Du(x)\rangle\,\nu(dx)<\infty.
$$

%%%%%%%%%%%%%%%%%%%%%%%%%%%%%%%%%%%%%%%%%%%%%%%%%%%%
\subsection{Weak solutions}
%%%%%%%%%%%%%%%%%%%%%%%%%%%%%%%%%%%%%%%%%%%%%%%%%%%%

We consider a Kolmogorov  operator defined on ${\mathcal F}{\mathcal C}^2_b(H)$ by 
\begin{equation}
\label{e1.6}
{\mathcal K}\varphi = \frac{1}{2}\;\Tr\; [ D^2\varphi]+ \frac{1}{2}\, \langle   x,Q^{ -1}D\varphi\rangle  - \langle DU(x),  D\varphi\rangle .
\end{equation}
Using the partial derivatives $D_k $ and $D_{kk}$, ${\mathcal K}$ may be rewritten as 
$${\mathcal K}   \varphi(x) = \frac{1}{2}\sum_{k=1}^{\infty}   D_{kk}\varphi(x) -\frac{1}{2} \sum_{k=1}^{\infty}
\lambda_k^{-1}  x_k D_{k}\varphi(x) -  \sum_{k=1}^{\infty}
 D_kU(x) D_{k}\varphi(x).$$

The measure $\nu$ enjoys the following important symmetrizing property:

\begin{Proposition}
\label{p2.1}
For all $\varphi \in   {\mathcal F}{\mathcal C}^2_b(H)$, $\psi \in {\mathcal F}{\mathcal C}^1_b(H)$ we have
 \begin{equation}
 \label{e2.1}
 \int_H\mathcal  K\varphi\,\psi\,d\nu  = -\frac12\int_H\langle  D\varphi, D\psi\rangle
d\nu.
 \end{equation} 
\end{Proposition} 
\begin{proof} Recalling \eqref{e1.9} we get
$$
\begin{array}{l}
\ds\frac{1}{2}\int_H   \sum_{k=1}^{\infty}  D_{kk}\varphi(x)\,\psi(x)\,d\nu= - \frac{1}{2} \int_H \sum_{k=1}^{\infty}  D_{k}\varphi(x)D_k \psi(x) \,d\nu
\\
\\
\ds + \int_H 
 \sum_{k=1}^{\infty} ( D_kU(x) D_k\varphi(x) + \frac{1}{2\lambda_k} x_k D_k\varphi(x) )
  \,d\nu ,
\end{array}
$$
and the conclusion follows (note that all series are finite sums in our case). \end{proof}

 Let $f\in L^2(H,\nu)$, $\lambda>0$. Taking into account formula 
 \eqref{e2.1}, we say that $u\in W^{1,2} (H,\nu)$ is a weak solution of   equation \eqref{e1.1}
 if we have
\begin{equation}
 \label{e2.2}
\lambda\int_H u\,\varphi\,d\nu+\frac12\int_H\langle Du,D\varphi\rangle\,d\nu=\int_Hf\,\varphi\,d\nu,\quad \forall\varphi\in W^{1,2} (H,\nu).
 \end{equation}
Since $ {\mathcal F}{\mathcal C}^1_b(H)$ is dense in  $W^{1,2} (H,\nu)$, it is enough that the above equality is satisfied for  every $\varphi\in {\mathcal F}{\mathcal C}^1_b(H)$. 

The function ${\mathcal A} : (W^{1,2} (H,\nu))^2\mapsto \R$, 
${\mathcal A} (u,\varphi )=  \lambda\int_H u\,\varphi\,d\nu+\frac12\int_H\langle Du,D\varphi\rangle\,d\nu$ is bilinear, continuous and coercive, while the function
$F: W^{1,2} (H,\nu) \mapsto \R$,  $F(\varphi) =  \int_Hf\,\varphi\,d\nu$, is linear and continuous. By the Lax--Milgram theorem there exists a unique 
$u\in W^{1,2} (H,\nu)$ such that ${\mathcal A} (u,\varphi )= F(\varphi)$ for each $\varphi\in W^{1,2} (H,\nu)$; namely 
equation \eqref{e1.1} has a unique weak solution $u\in W^{1,2} (H,\nu)$.

We denote by  $K :D(K) \subset L^2(H, \nu)\mapsto L^2(H, \nu)$ the operator associated to the quadratic form  $\mathcal A$ in $W^{1,2} (H,\nu)$. So, the domain $D(K )$  consists of all $u\in W^{1,2} (H,\nu)$ such that there exists $v\in L^2(H, \nu)$ satisfying 
 $$\frac{1}{2}\int_H\langle  Du, D\varphi\rangle\,d\nu = - \langle v, \varphi\rangle_{L^2(H, \nu)}$$
 for all $\varphi\in  W^{1,2} (H,\nu)$, or equivalently for all $\varphi \in {\mathcal F}{\mathcal C}^1_b(H)$. 
In this case, $v =  Ku$. 
The weak solution $u$ to  \eqref{e1.1} belongs to $D(K)$, and it is just $(\lambda I- K)^{-1}f$.

 \begin{Remark}
 \label{r2.2}
 \em We have  $ {\mathcal F}{\mathcal C}^2_b(H) \subset  D(K )$. In fact, for $u\in {\mathcal F}{\mathcal C}^2_b(H) $, integrating by parts we obtain  
\begin{equation}
\label{e2.3}
\frac{1}{2} \int_{H} \langle  Du,  D\varphi\rangle \,d\nu = -  \int_H  (\mathcal K u(x)) \varphi(x) \nu(dx),
\end{equation}
for all $  \varphi \in  {\mathcal F}{\mathcal C}^1_b(H)$. Here $ \mathcal K  u \in L^2(H,\nu) $ since it consists of the sum of a finite number of addenda, each of them in $L^2(H,\nu) $. Hence, $u\in D(K )$ and $K  u =  \mathcal K u$.  
\end{Remark}

To study the domain of $K$ it is convenient to introduce a family of  approximating problems, with  $U$ replaced by its Moreau--Yosida approximations $U_\alpha$ defined in 
 \eqref{e1.8}. 
Since $DU_{\alpha}$ is Lipschitz continuous, in the next section we consider the case of functions $U$ with Lipschitz gradient.

%%%%%%%%%%%%%%%%%%%%%%%%%%%%%%%%%%%%%%%%%%%%%%%%%
\subsection{The case of Lipschitz continuous $DU$}
\label{sect:Lipschitz}
%%%%%%%%%%%%%%%%%%%%%%%%%%%%%%%%%%%%%%%%%%%%%%%%%

Here we assume that  $U:H\mapsto \R$ is a differentiable convex function bounded from below and with Lipschitz continuous gradient. Since $DU$ is Lipschitz, it has at most linear growth, and $U$ has at most quadratic growth. Therefore, it satisfies Hypothesis \ref{Hyp}. 

The aim of this section is to show that for every $f\in L^2(H, \nu)$ the weak solution to \eqref{e1.1} belongs to $W^{2,2}(H, \nu) $ $\cap$ $W^{1,2}_{-1/2}(H, \nu)$ and   the estimate
 \begin{equation}
\label{e2.4}
 \lambda\int_H | Du|^2d\nu +\frac12\int_H \;Tr\;[ (D^2u )^2]d\nu 
  +\int_H\|Q^{-1/2 }Du\|^2d\nu +\int_H \langle D^2U  Du, Du \rangle d\nu 
  \le 4\int_H f^2d\nu  
\end{equation}
holds. 

Note that $U\notin W^{2,2}(H, \mu)$ in general. The term $\langle D^2U  Du, Du \rangle$ in the last integral is meant as follows: since $H$ is separable, and $\mu$ is non degenerate, by \cite[Theorem 6]{Ph} 
$DU:H\mapsto H$ is Gateaux differentiable $\nu$ almost everywhere. The Gateaux second order derivatives $D_{hk}U$ are bounded by a constant independent of $h$, $k$, since  $DU$ is Lipschitz continuous so that the Lipschitz constant of each $D_kU$ is bounded by a constant independent of $k$. Since $u\in W^{1,2}_{-1/2}(H, \nu)$ the double series  $\sum_{h,k} D_{hk}U  D_hu D_ku$ is well defined and belongs to $L^1(H, \nu)$. Indeed, 
$$\bigg|\sum_{h,k=1}^{\infty} D_{hk}U  D_hu D_ku\bigg| \leq C\bigg(\sum_{k=1}^{\infty}|D_ku|\bigg)^2 = C\bigg(\sum_{k=1}^{\infty}\lambda_k^{-1/2}|D_ku| \lambda_k^{1/2}\bigg)^2\leq C\|Q^{-1/2 }Du\|^2\,\mbox{\rm Tr}\,Q .$$
  
Moreover, we shall show that the weak solution is also a strong solution in the Friedrichs  sense. 

\begin{Definition}
A function $u\in L^2(H, \nu)$ is called {\em strong solution} (in the Friedrichs  sense) to \eqref{e1.1} if
there is a sequence $(u_n)$ of ${\mathcal F}{\mathcal C}^2_b(H)$ functions   that converge to $u$ in $L^2(H, \nu)$ and such that $\lambda u_n - {\mathcal K} u_n \to f$ in $L^2(H, \nu)$. 
\end{Definition}

\vspace{3mm}

In fact, we begin with the strong solution. The procedure  is the following: we show that the operator ${\mathcal K}  : {\mathcal F}{\mathcal C}^3_b(H)\mapsto L^2(H, \nu)$ is dissipative, so that it is closable. Then we show that   $(\lambda   - {\mathcal K} )({\mathcal F}{\mathcal C}^3_b(H))$ is dense in $L^2(H, \nu)$ for every $\lambda >0$. This implies that the closure $\overline{{\mathcal K} }$ of ${\mathcal K} $ generates a contraction semigroup in $L^2(H, \nu)$, and ${\mathcal F}{\mathcal C}^3_b(H)$ is a core, that is, it is dense in $D(\overline{{\mathcal K} })$ endowed with the graph norm.  In particular, for every $f\in L^2(H, \nu)$  and $\lambda >0$, equation \eqref{e1.1} has a unique  solution $u\in D(\overline{{\mathcal K} })$, which is a strong solution by definition. Then we show that $D(\overline{{\mathcal K} } )\subset W^{2,2} (H, \nu)$ and that \eqref{e2.4} holds. 
Eventually, we prove that the strong solution coincides with the weak solution. 

%%%%%%%%%%%%%%%%%%%%%%%%%%%%%%%%%%%%%%%%%%%%%%%%%
\subsubsection{${\mathcal K}  : {\mathcal F}{\mathcal C}^3_b(H)\mapsto L^2(H, \nu)$ is dissipative.}
%%%%%%%%%%%%%%%%%%%%%%%%%%%%%%%%%%%%%%%%%%%%%%%%%

This is just a simple consequence of the integration formula \eqref{e2.3}, taking   
$u =\varphi \in  {\mathcal F}{\mathcal C}^3_b(H)$.

%%%%%%%%%%%%%%%%%%%%%%%%%%%%%%%%%%%%%%%%%%%%%%%%%%%%%
\subsubsection{$(\lambda  I - {\mathcal K} )({\mathcal F}{\mathcal C}^3_b(H))$ is dense in $L^2(H, \nu)$.} 
\label{core}
 %%%%%%%%%%%%%%%%%%%%%%%%%%%%%%%%%%%%%%%%%%%%%%%%%%%%

We shall approach every element $f\in {\mathcal F}{\mathcal C}^{\infty}_b(H)$ by functions $g$  of the type $  g =\lambda v - {\mathcal K} v $, first with with $v \in {\mathcal F}{\mathcal C}^2_b(H)$ and then with with $v \in {\mathcal F}{\mathcal C}^3_b(H)$. This will be done using existence and regularity results for differential equations in finite dimensions. Since  ${\mathcal F}{\mathcal C}^{\infty}_b(H)$ is dense in $L^2(H, \nu)$, our aim will be achieved.

We recall that  $P_n$ is the orthogonal projection on the linear span of $e_1, \ldots, e_n$. 
We identify $P_n(H)$ with $\R^n$, by the obvious isomorphism $\R^n\mapsto P_n(H)$, $\xi \mapsto \sum_{k=1}^n \xi_k e_k$. 
The induced Gaussian measure in $\R^n$ is just ${\mathcal N}_{\;0, Q_n}$ where $Q_n=$ diag$(\lambda_1, \ldots, \lambda_n)$.

For any  function $v:H\mapsto \R$ we identify $v\circ P_n$   with the function $v_n:\R^n\mapsto \R$, $v_n(\xi): = v(\sum_{k=1}^n \xi_k e_k)$. 
In particular, we identify $U\circ P_n:H\mapsto \R$   with the function
$U_n: \R^n\mapsto \R$, $U_n(\xi) := U(\sum_{k=1}^n \xi_k e_k)$. $U_n$ is convex and $DU_n$ is Lipschitz continuous, and hence $U_n$ belongs to $W^{2, \infty}(\R^n, d\xi) \subset W^{2, \infty}(\R^n, {\mathcal N}_{\;0, Q_n})$.

For $\lambda >0$  let us consider the problem 
\begin{equation}
\label{e2.50}
\lambda v_n - {\mathcal L}  v_n + \langle DU_{n}, Dv_n \rangle = f_n, 
\end{equation}
where  the Ornstein--Uhlenbeck operator ${\mathcal L}$ in $\R^n$ is defined by 
$${\mathcal L} \varphi (\xi) =  \frac{1}{2} \sum_{k=1}^n ( D_{kk} \varphi(\xi) -  \lambda_k^{-1}\xi_k D_k \varphi(\xi)), \quad \xi \in \R^n.$$
Since $DU_n$ is Lipschitz continuous,   \eqref{e2.50}  has a  unique solution  $v_n \in \bigcup_{\alpha\in (0, 1)}$ $C^{2+\alpha}_b(\R^n)$.  A reference   is \cite[Theorem 1]{LV1}. 
In fact \cite[Theorem 1]{LV1} deals with large 
$\lambda$'s, but a standard application of the maximum principle (e.g., \cite[Lemma 2.4]{LV1}) and of the Schauder estimates of \cite[Theorem
1]{LV1} show that \eqref{e2.50} is uniquely solvable in $C^{2+\theta}_{b}(\R^n)$
for each $\lambda >0$. Moreover, 
an estimate for the first order derivatives of $v_n$, 
\begin{equation}
\label{e2.7n}
\|\,|Dv_n |\,\|_{\infty} \leq \frac{1}{ \lambda }\| \,|Df_n|\,\|_{\infty} ,  
\end{equation}
follows from the well known probabilistic representation formula for $v_n$, 
\begin{equation}
\label{Kolm_v}
v_n (\xi) = \int_0^{\infty}e^{-\lambda t} \E (f (X_n(t,\xi)))dt, \quad \xi \in \R^n, 
\end{equation}
$X_n(t,\xi)$ being the solution  to the stochastic ode in $\R^n$
$$\left\{\begin{array}{l}
dX_n(t,\xi)= -\frac{1}{2} Q_n^{-1} X_n(t,\xi)dt -   DU_n(X_n(t,\xi))dt +  dW_n(t),
\\
\\
X_n(0,\xi)=\xi,
\end{array}\right. $$
where $W_n(t) = P_nW(t)$ is a standard Brownian motion in $\R^n$. 
Indeed,    \eqref{e2.7n} follows taking into account that 
 $$ d(X_n(t,x) - X_n(t,y)) =  -\frac{1}{2} (Q_n^{-1}(X_n(t,x) - X_n(t,y)) dt - (DU_{n}(X_n(t,x)) - DU_{n}(X_n(t,y))dt $$
so that $X_n(\cdot ,x) - X_n(\cdot ,y)$ is almost surely differentiable, and taking the scalar product by $ X_n(t ,x) - X_n(t ,y)$ we get
  $\frac{d}{dt} \| X_n(t ,x) - X_n(t ,y)\|^2 \leq 0$,  by the monotonicity  of $D U_{n}$. This implies $ \|X_n(t ,x) - X_n(t ,y) \|\leq \|x-y\|$ and consequently  $|v_n^{\eps}(x) - v_n^{\eps}(y)| \leq \|f_n\|_{Lip}  \|x-y\|/\lambda$. 

Going  back to infinite dimensions, we set
\begin{equation}
\label{def}
V_n(x) := v_n (x_1, \ldots, x_n),  \quad x\in H.
\end{equation}
Then $V_n \in {\mathcal F}{\mathcal C}^2_b(H)$,   and  
\begin{equation}
\label{Vn}
\lambda V_n  - {\mathcal K}V_n  = f\circ P_n  + \langle DU - D(U\circ P_n), DV_n\rangle,
\end{equation}
where $f\circ P_n = f$ for $n$ large enough, since $f$ is cylindrical. The right-hand  side converges to $f$ as $n\to\infty$ since estimate 
 \eqref{e2.7n} implies
$$|\langle DU(x) - D(U\circ P_n)(x), DV_n^{\eps}(x)\rangle| \leq \frac{1}{\lambda}  \sup_{y\in H} \|Df(y)\| \, \|DU(x) - D(U\circ P_n)(x)\| $$
which goes to $0$ pointwise, since $DU$ is continuous, and in $L^2(H, \nu)$ by the dominated convergence theorem, since 
 $$\|D(U\circ P_n)(x)\| \leq   [DU]_{Lip} \|P_nx\| + \|DU(0)\|  \leq   [DU]_{Lip} \| x\| + \|DU(0)\| , $$
for each $n\in \N$. Therefore, $\lambda V_n  - {\mathcal K}V_n  $ converges to $f$ in $L^2(H, \nu)$, which implies that $(\lambda  I - {\mathcal K} )({\mathcal F}{\mathcal C}^2_b(H))$ is dense in $L^2(H, \nu)$.

This will be used later, in the proof of Proposition \ref{KolmLip}; however, it is not enough for our aims. This is because next formula \eqref{e2.11}, which is the starting point of all our optimal estimates,  is obtained differentiating  $\lambda u - {\mathcal K}u  $ for a cylindrical $u$, and  we need that $u$ has third order derivatives. So, we shall  approximate using  ${\mathcal F}{\mathcal C}^3_b$ functions instead of only ${\mathcal F}{\mathcal C}^2_b$ functions.

To be able to use regularity  theorems for elliptic equations in $\R^n$  that yield $C^3$ solutions, we need  regular coefficients, so we approach $U_n$ in a standard way by convolution with smooth mollifiers. Precisely, we fix once and for all a function $\theta\in C^{\infty}_{c}(\R^n)$ with support contained in the ball $B(0,1)$ of center $0$ and radius $1$, such that $\int_{\R^n}\theta(\xi)d\xi =1$, and for $\eps >0$ we set 
$$U^{\eps}_{n}(\xi) = \int_{\R^n} U_n(\xi -\eps y)\theta(y)dy,  \quad \xi\in \R^n.$$
Then $U^{\eps}_{n}$ is smooth and convex, and $DU^{\eps}_{n}$ is  Lipschitz continuous. Moreover, 
\begin{equation}
\label{e2.8}
\begin{array}{l}
\ds  |DU_n(\xi) - D{\mathcal U}^{\eps}_n(\xi)| = \bigg| \int_{\R^n} (DU_n(\xi) - DU_n(\xi-\eps y)) \theta(y)dy\bigg|
\\
\\
\ds \leq \eps [DU_n]_{Lip} \int_{\R^n}|y| \theta(y)dy \leq   \eps [DU_n]_{Lip} \leq \eps [DU]_{Lip}, \quad \xi\in \R^n.  
\end{array}  
\end{equation}

For $\lambda >0$ and $\eps >0$ let us consider the problem
\begin{equation}
\label{e2.5}
\lambda v_n^{\eps} - {\mathcal L}  v_n^{\eps} + \langle DU^{\eps}_{n}, Dv_n^{\eps}\rangle = f_n. 
\end{equation}
As before, since  $DU^{\eps}_{n}$ are Lipschitz continuous,   \eqref{e2.5} has a  unique solution $v_n^{\eps}\in \bigcup_{\alpha\in (0, 1)}$ $C^{2+\alpha}_b(\R^n)$, again by  \cite[Theorem 1]{LV1}. 
The functions $v_n^{\eps}$ are represented by
\begin{equation}
\label{Kolm_vn}
v_n^{\eps}(x) = \int_0^{\infty}e^{-\lambda t} \E (f_n(X^{\eps}(t,x)))dt,
\end{equation}
where  $X^{\eps}(t,x)$ is  the solution  to the stochastic ode 
$$\left\{\begin{array}{l}
dX^{\eps} (t,x)= -\frac{1}{2} Q_n^{-1} X^{\eps} (t,x)dt -   DU^{\eps}_{n}(X^{\eps} (t,x))dt +  dW_n(t),
\\
\\
X^{\eps} (0,x)=x,
\end{array}\right. $$
and $W_n(t)$ is a  standard Brownian motion in $\R^n$.
The representation formula \eqref{Kolm_vn} yields the sup norm estimates 
\begin{equation}
\label{e2.6}
\|v_n^{\eps}\|_{\infty} \leq \frac{1}{ \lambda }\| f_n\|_{\infty} , 
\end{equation}
\begin{equation}
\label{e2.7}
\|\,|Dv_n^{\eps}|\,\|_{\infty} \leq \frac{1}{ \lambda }\| \,|Df_n|\,\|_{\infty} . 
\end{equation}
  \eqref{e2.6} is immediate, while   \eqref{e2.7} follows arguing as in the proof of  \eqref{e2.7n}, since $D U^{\eps}_{n}$ is monotonic as well. 

We want to show that $v_n^{\eps}\in  C^{3}_b(\R^n)$. Since $DU^{\eps}_{n}$ is smooth, then $v_n^{\eps}$ belongs to $C^{\infty}(\R^n)$ by local elliptic regularity, and we need only to prove that its third order derivatives are bounded.  To this end we differentiate both sides of \eqref{e2.5} with respect to $x_i$, getting 
$$\lambda D_iv_n^{\eps} - {\mathcal L} D_iv_n^{\eps} + \frac{1}{\lambda_i} D_iv_n^{\eps} + \langle DU^{\eps}_{n}, D(D_iv_n^{\eps})\rangle = D_i f_n - \langle D(D_iU^{\eps}_{n}), Dv_n^{\eps}\rangle . $$
The right-hand  side is H\"older continuous and bounded. Applying once again the Schauder Theorem 
\cite[Theorem1]{LV1} we obtain $D_iv_n^{\eps}\in C^{2+\alpha}_b(\R^n)$ for each $\alpha \in (0,1)$. In particular, 
$v_n^{\eps}\in C^3_b(\R^n)$.

Let us go back to infinite dimensions and set 
\begin{equation}
\label{def-eps}
 V_n^{\eps}(x) := v_n^{\eps}(x_1, \ldots, x_n), \;\;  {\mathcal U}^{\eps}_n(x) =  U^{\eps}_{n}(x_1, \ldots, x_n), \quad x\in H.
\end{equation}
Then  $V_n^{\eps}\in {\mathcal F}{\mathcal C}^3_b(H)$ and  
\begin{equation}
\label{Vneps}
\lambda V_n^{\eps} - {\mathcal K}V_n^{\eps} = 
f\circ P_n + \langle DU - D{\mathcal U}^{\eps}_n, DV_n\rangle  .
\end{equation}
Concerning the right-hand  side, taking into account \eqref{e2.7} 
and \eqref{e2.8}, we get 
$$\begin{array}{l}
|\langle DU(x) - D{\mathcal U}^{\eps}_n(x), DV_n^{\eps}(x)\rangle|
\\
\\
\leq \frac{1}{\lambda}  \sup_{y\in H} \|Df(y)\| ( \|DU(x) - D(U\circ P_n)(x)\| +  \| D(U\circ P_n)(x)-  D{\mathcal U}^{\eps}_n(x)\|
\\
\\
\leq  \frac{1}{\lambda}  \sup_{y\in H} \|Df(y)\| ( \|DU(x) - D(U\circ P_n)(x)\| + \eps [DU]_{Lip(X)})
\end{array}$$
so that 
$$\begin{array}{l}
\ds \| \langle DU - D{\mathcal U}^{\eps}_n, DV_n^{\eps}\rangle \|_{L^2(H, \nu)}^2 \leq \bigg(\frac{1}{\lambda}  \sup_{y\in H} \|Df(y)\| \bigg)^2 \cdot 
\\
\\
\ds \cdot  2\bigg(\int_H  \|DU - D(U\circ P_n)\| ^2d\nu  + (\eps [DU]_{Lip(X)})^2 \bigg), \end{array}$$
where the first integral  $\int_H\|DU - D(U\circ P_n)\|^2d\nu$ vanishes as $n\to \infty$, as we already remarked. 
Therefore, $\| \langle DU - D{\mathcal U}^{\eps}_n, DV_n^{\eps}\rangle \|_{L^2(H, \nu)}$ is as small as we wish provided we take $n$ large and $\eps$ small, and the same holds for $\lambda V_n^{\eps} - {\mathcal K}V_n^{\eps} - f$.

Summarizing, we have proved the following proposition.

\begin{Proposition}
\label{p2.3}
The closure $\overline{\mathcal K}$ of the operator ${\mathcal K}: {\mathcal F}{\mathcal C}^3_b(H)\mapsto L^2(H, \nu)$ is $m$-dissipative, so that it generates a strongly continuous contraction semigroup in $L^2(H, \nu)$. 
In particular, for every $\lambda >0$ and $f\in L^2(H, \nu)$ problem \eqref{e1.1} has a unique strong solution $u$, that is: there is a sequence $(u_n)\subset  {\mathcal F}{\mathcal C}^3_b(H)$ such that 
$u_n \to u$ and $\lambda u_n - {\mathcal K}u_n\to f$ in $L^2(H, \nu)$. 
\end{Proposition}

%%%%%%%%%%%%%%%%%%%%%%%%%%%%%%%%%%%%%%%%%%%%
\subsubsection{$W^{2,2}(H, \nu)$ regularity of the strong solution and other estimates.}
\label{stimeottimali}
%%%%%%%%%%%%%%%%%%%%%%%%%%%%%%%%%%%%%%%%%%%%

To prove our estimates it is sufficient to consider functions  $u\in {\mathcal F}{\mathcal C}^3_b(H)$, which  is dense in the domain of $\overline{\mathcal K}$. So, we fix $u\in {\mathcal F}{\mathcal C}^3_b(H)$, $\lambda >0$,  and we set
$$\lambda u - {\mathcal K}u = f.$$
Estimates on $u$ and on $Du$ in terms of $f$ are elementary. They are obtained multiplying both sides by $u$ and taking into account \eqref{e2.1}. 

\begin{Lemma}
\label{l2.4}
We have 
$$\int_H (\lambda u^2 +\frac{1}{2}\|Du\|^2)d\nu = \int_H uf\,d\nu,$$
 and therefore
 \begin{equation}
\label{e2.9}
\int_H u^2d\nu \le \frac1{\lambda^2}\;\int_H f^2d\nu 
\end{equation}
and
\begin{equation}
\label{e2.10}
\int_H\|Du\|^2d\nu \le \frac{2}{\lambda}\;\int_H f^2d\nu .
\end{equation}
\end{Lemma}

Estimates on the second order derivatives are  less obvious. They are a consequence of the following proposition. 

\begin{Proposition}
\label{p2.5}
For each $u\in {\mathcal F}{\mathcal C}^3_b(H)$ we have
 \begin{equation}
 \begin{array}{l}
\label{e2.11}
\ds\lambda\int_H \|Du\|^2d\nu + \frac{1}{2} \int_H \Tr\;[(D^2u)^2]d\nu + \frac{1}{2} \int_H\|Q^{-1/2}Du\|^2d\nu 
\\
\\
\ds +\int_H \langle D^2UDu, Du \rangle d\nu 
= \int_H \langle Du,  Df \rangle d\nu  = 2 \int_H (\lambda u - f)f \,d\nu. 
\end{array} 
\end{equation}
\end{Proposition}
\begin{proof}
As in Section \ref{core}, we differentiate the equality $\lambda u - {\mathcal K}u = f$ with respect to $x_i$, then we multiply by $D_iu$ and sum up. We obtain 
$$\lambda  \|Du\|^2 - \sum_{i=1}^{\infty} ({\mathcal K}D_iu )D_iu + \sum_{i=1}^{\infty}\frac{(D_iu)^2}{2 \lambda_i}+ \sum_{i, j=1}^{\infty} D_{ij}UD_iuD_ju = \langle Df, Du\rangle , $$
where the series are in fact finite sums. Integrating on $H$ and taking \eqref{p2.1} into account, \eqref{e2.11} follows.  
\end{proof}

As a corollary of Lemma \ref{l2.4} and Proposition \ref{p2.5} we obtain estimates on the strong solution to \eqref{e1.1}. 

\begin{Proposition}
\label{p2.6}
Let $\lambda >0$, $f\in L^2(H, \nu)$ and  let $u$ be the strong solution to \eqref{e1.1}. 
Then $u\in W^{2,2}(H, \nu)\cap W^{1,2}_{-1/2}(H, \nu)$, and 
 \begin{equation}
\label{e2.12}
  \lambda\int_H \|Du\|^2d\nu + \frac{1}{2} \int_H \Tr\;[(D^2u)^2]d\nu + \frac{1}{2} \int_H\|Q^{-1/2}Du\|^2d\nu 
  +\int_H \langle D^2UDu, Du \rangle d\nu 
\leq 4 \int_H f^2 \,d\nu. 
\end{equation}
In addition, if $f\in {\mathcal F}{\mathcal C}^{\infty}_b(H)$, then $u$ is $\nu$--essentially bounded, and we have
\begin{equation}
\label{e2.13}
\ds\mbox{\rm ess} \sup_{x\in H} |u(x)| \leq \frac{1}{\lambda} \sup_{x\in H} |f(x)| .
\end{equation}
\end{Proposition}
\begin{proof} Let $u_j \in  {\mathcal F}{\mathcal C}^3_b(H)$ approach $u$ in $D(\overline{{\mathcal K}})$. By estimate \eqref{e2.10}, $Du_j \to Du$ in $L^2(H,\nu;H)$. By Proposition  \ref{p2.5}, equality 
 \eqref{e2.11} holds, with $u_j$ replacing $u$, and  $f_j := \lambda u_j - {\mathcal K}u_j$ replacing $ f$. Then, 
$$ \begin{array}{l}
\ds \lambda\int_H \|Du_j\|^2d\nu + \frac{1}{2} \int_H \Tr\;[(D^2u_j)^2]d\nu + \frac{1}{2} \int_H\|Q^{-1/2}Du_j\|^2d\nu 
\\
\\
\ds +\int_H \langle D^2UDu_j, Du_j \rangle d\nu 
\leq 2 \int_H (\lambda u_j - f_j)f_j \,d\nu \leq 4\|f_j\|^2_{L^2(H, \nu)}, 
\end{array} $$
while by \eqref{e2.9} we have $\lambda \|u_j\|_{L^2(H, \nu)}\leq  \|f_j\|_{L^2(H, \nu)}$. 
Since $f_j\to f$ in $L^2(H, \nu)$ as $j\to \infty$,   $(u_j)$ is a Cauchy sequence in $W^{2,2}(H, \nu)$ and in  $W^{1,2}_{-1/2}(H, \nu)$. 
So, $u$ belongs to such spaces, and letting $j\to \infty$ estimate \eqref{e2.12} follows. 

To prove the last statement,  for $f\in {\mathcal F}{\mathcal C}^{\infty}_b(H)$
we approach $u$ by the functions used in the proof of Proposition \ref{p2.3}. 
Then
\eqref{e2.13}  follows from \eqref{e2.6}, taking into account that for a suitable sequence $(j_k)$, $(u_{j_k})$ converges to $u$, $\nu$-a.e. 
\end{proof}

%%%%%%%%%%%%%%%%%%%%%%%%%%%%%%%%%%%%%%%%%%%%%%%
\subsubsection{Weak = strong.}
\label{Weak = strong}
%%%%%%%%%%%%%%%%%%%%%%%%%%%%%%%%%%%%%%%%%%%%%%%

For $\lambda >0$ and $f\in L^2(H, \nu)$ let $u$ be the strong solution to \eqref{e1.1} given by Proposition \ref{p2.3}. Let $u_n\in  {\mathcal F}{\mathcal C}^3_b(H)$ be such that 
$u_n \to u$ and $f_n:= \lambda u_n - {\mathcal K}u_n \to f$ in $L^2(H, \nu)$. As we remarked in the proof of Proposition \ref{p2.6}, $u_n \to u$ in $W^{1,2}(H, \nu)$. 
 
Fix $\varphi \in {\mathcal F}{\mathcal C}^1_b(H)$. Multiplying  both sides of  $ \lambda u_n - {\mathcal K}u_n = f_n$ by $\varphi$, integrating  over $H$ and recalling \eqref{e2.1}, we obtain
$$\lambda\int_H u_n\,\varphi\,d\nu+\frac12\int_H\langle Du_n,D\varphi\rangle\,d\nu=\int_Hf_n\,\varphi\,d\nu .$$
Letting $n\to \infty$ yields that $u$ is the weak solution to \eqref{e1.1}. So,   weak and   strong solutions to \eqref{e1.1} do coincide.  
 
 \vspace{2mm}

 As a consequence of coincidence of strong and weak solutions  we obtain a probabilistic representation formula for the weak solution to \eqref{e1.1}. 
Let $W(t)$ be any $H$-valued cylindrical Wiener process defined in a probability space $(\Omega, {\mathcal F}, \P)$. A construction of such a process may be found, for example,  in \cite[Section 4.3]{DPZrosso}. For each $x\in H$ consider the stochastic differential equation 
\begin{equation}
\label{stocastica}
dX=(AX-DU(X))dt+dW(t),\quad X(0)=x.  
\end{equation}
We recall that a mild solution to  \eqref{stocastica} is a ${\mathcal F}_t$ adapted, $H$-continuous process that satisfies
$$X(t) = e^{tA}x - \int_0^t e^{(t-s)A} DU(X(s))ds + \int_0^t e^{(t-s)A} dW(s), \quad t\geq 0, $$
where  ${\mathcal F}_t$ is the natural filtration of $W(t)$. Existence and uniqueness of a mild solution to \eqref{stocastica} follow, for example,  from \cite[Theorem 5.5.8]{DPZ2}; see also Remark 5.5.7 of \cite{DPZ2}.

\begin{Proposition}
\label{KolmLip}
For $\lambda >0$ and $f\in C_b(H)$, let $u$ be the weak solution to \eqref{e1.1}. Then
\begin{equation}
\label{finalmente}
u  = \int_0^{+\infty} e^{-\lambda t} \E f(X(t,\cdot))\,dt.
\end{equation}
\end{Proposition}
\begin{proof}
As a first step, let $f\in {\mathcal FC}^{\infty}_b(H)$,   let   $V_n$ be the functions defined in \eqref{def} and set 
$f_n:= \lambda V_n - K V_n $. In Section \ref{core} we have shown that $\lim_{n\to \infty}f_n= f$  in $L^2(H, \nu)$. Therefore, 
$u = R(\lambda, K)f = \lim_{n\to \infty}R(\lambda, K)f_n = \lim_{n\to \infty}V_n$. On the other hand, we have 
$V_n(x) = v_n(x_1, \ldots, x_n)$, where 
 the functions $v_n$ solve \eqref{e2.50}. 
This implies that $V_n$ satisfies
\begin{equation}
\label{KolmV}
V_n(x)  =  \int_0^{+\infty} e^{-\lambda t} \E f(X_n(t,x))\,dt , \quad x\in H ,
\end{equation}
where $X_n$ is the mild solution to 
\begin{equation}
\label{stocastica_n}
dX_n=(AX_n-D(U\circ P_n)(X_n))dt+dW(t),\quad X_n(0)=P_nx, 
\end{equation}
and for every $t>0$, $x\in X$ we have $\lim_{n\to\infty}X_n(t,x) = X(t,x)$, a.s.  
Letting $n\to \infty$ in \eqref{KolmV}, the left-hand  side goes to $u$ in $L^2(H, \nu)$. The right-hand  side converges to 
$ \int_0^{+\infty} e^{-\lambda t} \E f (X (t,x))\,dt $ pointwise and in $L^2(H, \nu)$ by the dominated convergence theorem. Indeed, for each $x\in H$ and $t>0$ 
we have $\lim_{n\to \infty}f (X_n(t,x)) = f(X(t,x))$ a.s.,  and $|f (X_n(t,x))| \leq  \|f\|_{\infty} $. 
Therefore, the statement holds if $f\in {\mathcal FC}^{\infty}_b(H)$. 

If $f\in C_b(H)$, 
it is possible to approach it, pointwise and in $L^2(H, \nu)$, by a sequence $(f_n)$ of functions belonging to ${\mathcal FC}^{\infty}_b(H)$. For instance, one can take approximations by convolution of $f\circ P_n$. Then, $u_n:=R(\lambda, K)f_n$ satisfy \eqref{finalmente} with $f$ replaced by $f_n$ and converge to $u=R(\lambda, K)f$ in $L^2(H, \nu)$. The right-hand  sides converge to $\int_0^{+\infty} e^{-\lambda t} \E f(X(t,\cdot))\,dt$  in $L^2(H, \nu)$, again by the dominated convergence theorem, and the statement follows. 
\end{proof}

%%%%%%%%%%%%%%%%%%%%%%%%%%%%%%%%%%%%%%%%%%%%%%%%%%%%
\subsection{The general case}
\label{sect:general}
%%%%%%%%%%%%%%%%%%%%%%%%%%%%%%%%%%%%%%%%%%%%%%%%%%%%

Here we apply the results of Section \ref{sect:Lipschitz} to prove our main result.

\begin{Theorem}
\label{t2.7}
Under Hypothesis  \ref{Hyp},  for every $\lambda >0$ and $f\in  L^2(H, \nu)$, the weak solution $u$ to \eqref{e1.1}  belongs to $W^{2,2}(H, \nu)$ $\cap $ $W^{1,2}_{-1/2}(H, \nu)$, and it satisfies 
\begin{equation}
\label{e2.17}
\int_H u ^2d\nu \le \frac1{\lambda^2}\;\int_H f^2\,d\nu  , \quad \int_H\|Du \|^2 d\nu  \le \frac{2}{\lambda}\;\int_H f^2\,d\nu , 
\end{equation} 
\begin{equation}
\label{e2.18}
\frac{1}{2} \int_H \Tr\;[(D^2u )^2]\,d\nu   + \int_H\|Q^{-1/2}Du\|^2d\nu 
\leq  4 \int_H f^2 \,d\nu. 
\end{equation} 
\end{Theorem}
\begin{proof} Let $U_{\alpha}$ be the Moreau--Yosida approximations of $U$, defined in \eqref{e1.8}. Since $DU_{\alpha}$ is Lipschitz continuous, we may use the results of Sections \ref{stimeottimali},   \ref{Weak = strong} for problem
\begin{equation}
\label{e2.14}
\lambda u_{\alpha}-{\mathcal L} u_\alpha+\langle DU_\alpha,Du_\alpha   \rangle=f . 
\end{equation}
Let $Z_{\alpha} = \int_H e^{-2U_{\alpha}(x)}\mu(dx)$ and $\nu_{\alpha}:= e^{-2U_{\alpha}}\mu /Z_{\alpha} $.   
Fix any $f\in {\mathcal F}{\mathcal C}^{\infty}_b(H)$, $\lambda >0$, and let $u_{\alpha}$ be the strong solution to \eqref{e2.14} in the space $L^2(H, \nu_{\alpha})$. By Lemma \ref{l2.4}, 
\begin{equation}
\label{e2.15}
\int_H u_{\alpha}^2e^{-2U_{\alpha}}d\mu  \le \frac1{\lambda^2}\;\int_H f^2e^{-2U_{\alpha}}d\mu  , \quad \int_H\|Du_{\alpha}\|^2e^{-2U_{\alpha}}d\mu  \le \frac{2}{\lambda}\;\int_H f^2e^{-2U_{\alpha}}d\mu , 
\end{equation}
and by Proposition \ref{p2.6}, 
\begin{equation}
\label{e2.16} 
\begin{array}{l}
\ds \frac{1}{2} \int_H \Tr\;[(D^2u_{\alpha})^2]e^{-2U_{\alpha}}d\mu +
\frac{1}{2} \int_H\|Q^{-1/2}Du_{\alpha}\|^2e^{-2U_{\alpha}}d\mu
\\
\\
\ds +  \int_H \langle D^2U_{\alpha}Du_{\alpha}, Du_{\alpha} \rangle e^{-2U_{\alpha}}d\mu
\leq  4 \int_H f^2 \,e^{-2U_{\alpha}}d\mu .
\end{array}
\end{equation}
The right-hand  sides of \eqref{e2.15} and \eqref{e2.16} are bounded by a constant independent of $\alpha$,   since $U_{\alpha} \geq \inf U$ so that   
\begin{equation}
\label{f}
\int_H f^2 \,e^{-2U_{\alpha}}d\mu \leq \|f\|_{\infty}^2 e^{-2\inf U}. 
\end{equation}
Since  $U_\alpha \le U$, then $e^{-2U}\leq e^{-2U_{\alpha}}$, and it follows that $u_{\alpha}\in W^{2,2}(H, \nu)$ and their $W^{2,2}(H, \nu)$ norms are bounded by a constant independent of $\alpha$. A sequence $(u_{\alpha_n})$, with $\lim_{n\to \infty}\alpha_n =0$,  converges weakly in $W^{2,2}(H, \nu)$ and   in $W^{1,2}_{-1/2}(H, \nu)$  to a limit function denoted by $u$. 
Letting $n\to \infty$ yields that $u$ satisfies \eqref{e2.17} and \eqref{e2.18}. 
Our aim is to show that $u$ coincides with the weak solution to 
\eqref{e1.1}. 
For every $n$ we have
$$\lambda\int_H u_{\alpha_n}\,\varphi\,e^{-2U_{\alpha_n}}d\mu + \frac{1}{2}  \int_H\langle Du_{\alpha_n},D\varphi\rangle\,e^{-2U_{\alpha_n}}d\mu =\int_Hf\,\varphi\,e^{-2U_{\alpha_n}}d\mu,\quad \varphi \in {\mathcal F}{\mathcal C}^1_b(H). $$
Letting $n\to \infty$, the right-hand  side  converges to  
$\int_Hf\,\varphi\,e^{-2U }d\mu$. Let us split the left-hand  side as 
$$\begin{array}{l}
\ds \int_H( \lambda u_{\alpha_n}\,\varphi\ + \frac{1}{2}  \langle Du_{\alpha_n},D\varphi\rangle) e^{-2U_{\alpha_n}}d\mu =
\\
\\
\ds =  \int_H( \lambda u_{\alpha_n}\,\varphi\ + \frac{1}{2}  \langle Du_{\alpha_n},D\varphi\rangle) e^{-2U}d\mu 
 + \int_H( \lambda u_{\alpha_n}\,\varphi\ + \frac{1}{2}  \langle Du_{\alpha_n},D\varphi\rangle) ( 1 -  e^{-2U +2U_{\alpha_n}} )e^{-2U_{\alpha_n}} d\mu .
\end{array}$$
The first integral converges to $\int_H( \lambda u \,\varphi\ + \frac{1}{2}  \langle Du ,D\varphi\rangle)e^{-2U}d\mu$. We claim that the second integral too vanishes as $n\to \infty$. Indeed, by the H\"older inequality with respect to the measure $  e^{-2U_{\alpha_n}} d\mu $,   its modulus  is bounded by 
$$\begin{array}{l}
\ds
\bigg( \int_H( \lambda u_{\alpha_n}\,\varphi\ +\frac12 \langle Du_{\alpha_n},D\varphi\rangle)^2 e^{-2U_{\alpha_n}}d\mu \bigg)^{1/2}
\bigg(  \int_H ( 1 -  e^{-2U +2U_{\alpha_n}} )^2e^{-2U_{\alpha_n}} d\mu   \bigg)^{1/2}
\\
\\
\ds \leq   \|\varphi\|_{C^1_b(H)} (\|\lambda u_{\alpha_n}\|_{L^2(H, e^{-2U_{\alpha_n}}\mu)} + \frac{1}{2} 
\| \,\|Du_{\alpha_n}\|\, \|_{L^2(H,e^{-2U_{\alpha_n}}\mu)})  \bigg(  \int_H ( 1 -  e^{-2U +2U_{\alpha_n}} )^2e^{-2U_{\alpha_n}} d\mu   \bigg)^{1/2}. 
\end{array}$$
Recalling \eqref{f}, \eqref{e2.15} implies now that 
$$\|\lambda u_{\alpha_n}\|_{L^2(H, e^{-2U_{\alpha_n}}\mu)} + \frac{1}{2} 
\| \,\|Du_{\alpha_n}\|\, \|_{L^2(H,e^{-2U_{\alpha_n}}\mu)}$$
 is bounded by a constant independent of $n$. Moreover $\int_H ( 1 -  e^{-2U +2U_{\alpha_n}} )^2e^{-2U_{\alpha_n}} d\mu $ vanishes as $n\to \infty$ by the dominated convergence theorem, and the claim is proved. 

Therefore, $u$ satisfies \eqref{e2.2} for every $\varphi \in {\mathcal F}{\mathcal C}^1_b(H)$, and hence it is the weak solution to \eqref{e1.1}. 

If $f\in  L^2(H, \nu)$, there is a sequence of ${\mathcal F}{\mathcal C}^{\infty}_b(H)$ functions that converge to $f$ in $L^2(H, \nu)$. 
The sequence $(R(\lambda, K)f_k)$ of the weak solutions to \eqref{e1.1}  with $f$ replaced by  $f_k$ converge to the weak solution $u = R(\lambda, K)f$ of  \eqref{e1.1}, and it is a Cauchy sequence in $W^{2,2}(H, \nu)$ and   in $W^{1,2}_{-1/2}(H, \nu)$ by estimate \eqref{e2.18}. 
Then $u\in  W^{2,2}(H, \nu)$ $\cap $ $W^{1,2}_{-1/2}(H, \nu)$, and it satisfies 
\eqref{e2.18} too. \end{proof}

\subsection{Another maximal estimate}

Under further assumptions we may recover the full estimate on $Du$ that holds in the case that $DU$ is Lipschitz continuous. In fact, we shall show below that 
\begin{equation}
\label{e2.19} 
 \int_H \langle D^2U\,Du, Du \rangle d\nu
\leq  4 \int_H f^2 d\nu , 
\end{equation}
in the case where $U\in C^2(H)$, while in Section 4.2 it will be proved in a specific example with  $U\notin C^2(H)$.  
Here and in the following, we denote by $C^2(H)$ the space of the twice Fr\'echet differentiable functions from $H$ to $\R$, with continuous second order derivative. 

We need  a preliminary result.

\begin{Lemma}
\label{l2.8}
Under Hypothesis  \ref{Hyp}, for each $f\in C_b(H)$ there is $\alpha_n\to 0$ such that $u_{\alpha_n}\to u$ in $W^{1,2}(H, \nu)$ as $n\to \infty$. 
\end{Lemma}
\begin{proof}
We already know  that there exists  a sequence  $(u_{\alpha_n})$ weakly convergent to  $u$ in $W^{1,2}(H, \nu)$. So, it is enough to show that
\begin{equation}
\label{e2.20}
\limsup_{n\to\infty}  |u_{\alpha_n} |_{W^{1,2}(H, \nu)} \leq   |u  |_{W^{1,2}(H, \nu)}.
\end{equation}
for some equivalent norm $|\cdot |_{W^{1,2}(H, \nu)} $ in $W^{1,2}(H, \nu)$. 

By Lemma \ref{l2.4} we have
$$\int_H (\lambda |u_{\alpha_n}|^2 + \frac{1}{2}\|Du_{\alpha_n}\|^2) e^{-2U_{\alpha_n}}d\mu
= \int_H f u_{\alpha_n}e^{-2U_{\alpha_n}}d\mu .$$
We claim that the right-hand  side converges  to $Z\int_H f u \, d\nu$ as
$n\to \infty$.  In fact we have 
$$ \int_H f  u_{\alpha_n}e^{-2U_{\alpha_n}} d\mu =  \int_H f u_{\alpha_n} e^{-2U }d\mu
+ \int_H f u_{\alpha_n} (1-  e^{2U_{\alpha_n}-2U })e^{-2U_{\alpha_n}}d\mu , $$
where the first addendum  tends to $Z\int_H f u  d\nu$, and the second one is estimated by 
$$ \left |\int_H f u_{\alpha_n} (1-  e^{2U_{\alpha_n}-2U })e^{-2U_{\alpha_n}}d\mu\right| \le \|f\|_{\infty}  \|u_{\alpha_n}\|_{L^2(H, e^{-2U_{\alpha_n}}\mu)}
\int_H (1-e^{ 2U_{\alpha_n} -2U })^2e^{-2U_{\alpha_n}}d\mu ,
$$
which vanishes as $n\to \infty$ because
$ \|u_{\alpha_n}\|_{L^2(H, e^{-2U_{\alpha_n}}\mu)}$ is bounded and 
$$\lim_{n\to \infty} \int_H (1-e^{ 2U_{\alpha_n} -2U })^2e^{-2U_{\alpha_n}}d\mu = 0$$     
by the the dominated convergence theorem. 

Therefore  we have
$$  \limsup_{n\to \infty}\int_H (\lambda u_{\alpha_n}^2 + \tfrac{1}{2}\|Du_{\alpha_n}\|^2)e^{-2U }d\mu
  \leq 
\limsup_{n\to \infty} \int_H (\lambda |u_{\alpha_n}|^2 + \tfrac{1}{2}\|Du_{\alpha_n}\|^2) e^{-2U_{\alpha_n}}d\mu 
  = Z\int_H f u \, d\nu .$$
Moreover
$$\int_H f u \, d\nu = \int_H (\lambda u^2 + \tfrac{1}{2}\|Du \|^2)d\nu,$$
so that
$$\limsup_{n\to \infty}
\int_H (\lambda |u_{\alpha_n}|^2 + \tfrac{1}{2}\|Du_{\alpha_n}\|^2)d\nu \leq 
\int_H (\lambda u^2 + \frac{1}{2}\|Du \|^2)d\nu ,$$
and \eqref{e2.20} follows. 
\end{proof}

Now we can  prove estimate \eqref{e2.19}.

\begin{Theorem}
\label{t2.10}
Let $U$ be a $C^2$ function satisfying Hypothesis \ref{Hyp}. Then  \eqref{e2.19} is fulfilled for all  $f\in   L^2(H, \nu)$. 
 \end{Theorem}
 \begin{proof}
Since $C_b(H)$ is dense in $L^2(H, \nu)$ it is sufficient to prove   \eqref{e2.19} when $f\in  C_b(H)$. In this case,  let  $\alpha_n\to 0$  be such that  $ u_{\alpha_n} \to u$ in $W^{1,2}(H, \nu)$ (Lemma \ref{l2.8}). Then
$Du_{\alpha_n} \to Du$ in $L^2(H, \nu; H)$  and so (possibly  replacing $(\alpha_n)$ by  a subsequence) $Du_{\alpha_n}(x) \to Du(x)$ for almost all $x$. Using Lemma \ref{l2.9}, for these $x$  we have
$$
\lim_{n\to \infty} \langle D^2U_{\alpha_n}(x) Du_{\alpha_n}(x), Du_{\alpha_n}(x)\rangle e^{-2U_{\alpha_n}(x)}=
\langle D^2U (x) Du (x), Du (x)\rangle  e^{-2U (x)}, 
$$
and by Fatou's Lemma, 
$$\begin{array}{l}
\ds \int_H \langle D^2U (x) Du (x), D (x)\rangle   d\nu 
 = \int_H \langle D^2U (x) Du (x), D (x)\rangle e^{-2U (x)} d\mu 
 \\
 \\
 \ds \leq \liminf_{n\to \infty} \int_H 
\langle D^2U_{\alpha_n}(x) Du_{\alpha_n}(x), Du_{\alpha_n}(x)\rangle e^{-2U_{\alpha_n}(x)}d\mu
\\
\\
\ds \leq 4 \liminf_{n\to \infty} \int_H f^2 \, e^{-2U_{\alpha_n}}d\mu = 4  \int_H f^2 \, d\nu .
\end{array}$$
\end{proof}

%%%%%%%%%%%%%%%%%%%%%%%%%%%%%%%%%%%%%%%%%%%%%%%%%%%
\section{Perturbations}
\label{Per}
%%%%%%%%%%%%%%%%%%%%%%%%%%%%%%%%%%%%%%%%%%%%%%%%%%%%

The regularity results and estimates of Section 3 open the way to new results for nonsymmetric Kolmogorov operators, by perturbation. Here we consider the operator $K_1 $ in the space $L^2(H, \nu)$ defined by 
\begin{equation}
\label{e3.1}
D(K_1) = D(K), \quad K_1v := Kv +\langle B(x), Dv(x)\rangle  
\end{equation}
with a (possibly) nongradient field  $B:H\mapsto H$. 

We shall give two perturbation results, the first one in the general case (Section 4.1) and the second one in the case where 
the weak solution to \eqref{e1.1} satisfies \eqref{e2.19}   (Section 4.2).  
In both cases we shall use the next proposition and a part of its proof.

\begin{Proposition}
\label{p3.1}
Let $A$ be a self-adjoint  dissipative operator in $L^2(H, \nu)$, and let  ${\mathcal B}: D(A)\mapsto L^2(H, \nu)$ be a linear operator such that 
\begin{equation}
\label{e3.2}
\|{\mathcal B}v\|_{L^2(H, \nu)}^2 \leq a\|Av\|_{L^2(H, \nu)}^2 + b \|v\|_{L^2(H, \nu)}^2, \quad v\in D(A), 
\end{equation}
for some $a<1/(\sqrt{2}+1)^2$  and $b>0$. Then the operator
$$A_1: D(A)\mapsto L^2(H, \nu), \quad A_1v = Av + {\mathcal B}v$$
generates an analytic semigroup in $L^2(H, \nu)$. 
\end{Proposition}
\begin{proof}
Let us denote by  ${\mathcal X} = L^2(H, \nu; \C)$ the complexification of $L^2(H, \nu)$ and by ${\mathcal A}$ the complexification of $A$, ${\mathcal A}(u+iv) = Au + iAv$. Then the spectrum of ${\mathcal A}$ is contained in $(-\infty, 0]$ and we have $\|\lambda R(\lambda, {\mathcal A})\|_{{\mathcal L}(\mathcal X)} \leq 1/\cos(\theta/2)$ for $\lambda \in \C\setminus (-\infty, 0]$, with $\theta = \arg \lambda$. Hence, for Re$\,\lambda >0$ we have $\|\lambda R(\lambda, {\mathcal A})\|_{{\mathcal L}(\mathcal X)} \leq \sqrt{2}$. 

A standard general perturbation result for analytic semigroups in Banach spaces  states that if the generator ${\mathcal A}$ of an analytic semigroup in a complex Banach space ${\mathcal X}$ satisfies $\|\lambda R(\lambda, {\mathcal A})\|_{{\mathcal L}(\mathcal X)} \leq M$ for Re$\,\lambda > \omega$, then for any linear perturbation ${\mathcal B}: D({\mathcal A})\mapsto \mathcal X$ that satisfies
$$\|{\mathcal B}v\|_{\mathcal X} \leq c_1\|Av\|_{\mathcal X} + c_2 \|v\|_{\mathcal X}, \quad v\in D({\mathcal A}), $$
with $c_1<1/(M+1)$ and  $c_2\in \R$, the sum ${\mathcal A}+ {\mathcal B}: D({\mathcal A})\mapsto \mathcal X$ generates an analytic semigroup in $\mathcal X$. We write down a proof, which will be used later. 

For Re$\,\lambda > \omega$ the resolvent equation $\lambda u - ({\mathcal A}+ {\mathcal B})u = f$ is equivalent (setting $\lambda u - {\mathcal A}u = v$) to 
the fixed point problem $v = Tv$, with $T:{\mathcal X}\mapsto {\mathcal X}$, $Tv = {\mathcal B}R(\lambda, {\mathcal A})v + f$. We have 
$$\| Tv\| \leq c_1 \| {\mathcal A}R(\lambda, A)v\| + c_2 \|  R(\lambda, {\mathcal A})v\| \leq c_1(M+1)\|v\| + \frac{c_2M}{|\lambda |}\|v\|, \quad v\in {\mathcal X}. $$
Fix $\omega_0>\omega$ such that $C:= c_1(M+1)+ c_2M/\omega_0 <1 $. Then for every $\lambda $ in the halfplane 
 Re$\,\lambda \geq  \omega_0$ $T$ is a contraction with constant $C$, the equation $v = Tv$ has a unique solution $v\in \mathcal X$ and $\|v\|\leq \|f\|/(1-C)$, and the resolvent equation $\lambda u - A_1u = f$ has a unique solution $u = R(\lambda, A)v$ with $\|u\|\leq M\|f\|/|\lambda| (1-C)$, and the statement follows. 

In our case we can take $\omega =0$ and $M= \sqrt{2}$. Assumption \eqref{e3.2} implies that $\|{\mathcal B}v\|_{\mathcal X} \leq \sqrt{a}\|Av\|_{\mathcal X} + \sqrt{b} \|v\|_{\mathcal X}$, for every $v\in D({\mathcal A})$, so we require $a<1/(\sqrt{2}+1)^2$. Once we know that ${\mathcal A}+ {\mathcal B}$ generates an analytic semigroup $T(t)$ in $L^2(H, \nu; \C)$, it is sufficient to remark that the restriction of $T(t)$ to  $L^2(H, \nu)$ preserves $L^2(H, \nu)$, and it is an analytic semigroup in $L^2(H, \nu)$. 
\end{proof}

%%%%%%%%%%%%%%%%%%%%%%%%%%%%%%%%%%%%%%%%%%%%%%%%%%%%
\subsection{First perturbation}
%%%%%%%%%%%%%%%%%%%%%%%%%%%%%%%%%%%%%%%%%%%%%%%%%%%%

\begin{Proposition}
\label{p3.2}
Let $U$ satisfy Hypothesis \ref{Hyp}. Let $B:H\mapsto H$ be $\mu$-measurable (hence, $\nu$-measurable) and such that there exist $c_1 \in (0,1/2(\sqrt{2}+1))$, $c_2>0$ such that for a.e. $x\in H$ we have
\begin{equation}
\label{e3.3}
|\langle B(x), y\rangle| \leq c_1 \|Q^{-1/2}y\| + c_2\|y\|, \quad y\in Q^{1/2}(H).
\end{equation}
Then the operator $K_1$ defined in \eqref{e3.1}  generates an analytic semigroup in $L^2(H, \nu)$. In particular, 
there exist  $\lambda_0\geq 0$, $C>0$ such that for every $\lambda>\lambda_0$ and for every $f\in L^2(H, \nu)$ the equation $\lambda v - K_1v =f$ has a unique solution $v\in D(K) $, and  
$$\|v\|_{D(K)} \leq C\|f\|_{L^2(H, \nu)}. $$
\end{Proposition}
\begin{proof}
In view of Proposition \ref{p3.1}, it is sufficient to show that the operator ${\mathcal B}$ defined in $D(K)$ by 
$${\mathcal B}u(x)= \langle B(x), Du(x)\rangle , \quad x\in H, $$
satisfies estimate 
\begin{equation}
\label{e3.2a}
\|{\mathcal B}v\|_{L^2(H, \nu)}^2 \leq a\|Kv\|_{L^2(H, \nu)}^2 + b \|v\|_{L^2(H, \nu)}^2, \quad v\in D(K),
\end{equation}
 for some  $a< (\sqrt{2}+1)^{-2}$. 
We note that for every $u\in D(K)$ we have
\begin{equation}
\label{e3.4}
\int_H \| Du\|^2 d\nu \leq 4\lambda \int_Hu^2d\nu + 
\frac{4}{\lambda} \int_H (Ku)^2d\nu, \quad \forall \lambda >0, 
\end{equation}
\begin{equation}
\label{e3.5}
 \int_H \|Q^{-1/2}Du\|^2 d\nu \leq 4 \int_H(Ku)^2d\nu .
\end{equation}
Estimate 
\eqref{e3.4} follows from  \eqref{e2.17}, taking $f= \lambda u - Ku$. Estimate \eqref{e3.5} follows from  \eqref{e2.18} taking again $f= \lambda u - Ku$, and letting $\lambda \to 0$.  Using \eqref{e3.4} and \eqref{e3.5}, for each  $\eps\in (0,1)$ and $\lambda >0$ we get 
$$\begin{array}{l}
\ds \int_H \langle B , Du \rangle ^2d\nu \leq  \int_H (c_1 \|Q^{-1/2}Du\| + c_2\|Du\|)^2 d\nu
\\
\\
\ds \leq c_1^2(1+\eps)\int_H   \|Q^{-1/2}Du\|^2d\nu + c_2^2\bigg(1+\frac{1}{\eps}\bigg)\int_H \|Du\|^2 d\nu
\\
\\
\ds \leq 4c_1^2(1+\eps)\int_H(Ku)^2d\nu + c_2^2\bigg(1+\frac{1}{\eps}\bigg)\bigg(4\lambda \int_Hu^2d\nu + 
\frac{4}{\lambda} \int_H (Ku)^2d\nu\bigg)
\end{array}$$
Since $4c_1^2< 1/(\sqrt{2}+1)^2$, there is $\eps >0$ such that $ 4c_1^2(1+\eps)< 1/(\sqrt{2}+1)^2$. Fixed such $\eps$, 
choose   $\lambda$ big enough, such that $a:= 4c_1^2(1+\eps) + 4c_2^2 (1+1/\eps)/\lambda < 1/(\sqrt{2}+1)^2$. 
With these choices estimate \eqref{e3.2a} is satisfied with $a  < 1/(\sqrt{2}+1)^2$, and the statement follows from Proposition \ref{p3.1}. 
\end{proof}

\begin{Remark}
{\em The assumptions of Proposition \ref{p3.2} are satisfied if $x\mapsto Q^{\alpha}B(x) \in L^{\infty}(H, \nu ; H)$ for some $\alpha <1/2$. Indeed, in this case for $y\in Q^{1/2}(H)$ and a.e. $x\in H$, we have}
$$|\langle B(x), y\rangle|  = |\langle Q^{ \alpha}B(x), Q^{-\alpha}y\rangle|
\leq  \|Q^{\alpha}B(\cdot)\|_{\infty} (\eps \|Q^{-1/2}y\| + c(\eps) \|y\|), \quad x\in H, \;\eps >0,  $$
{\em and choosing  $\eps$ small enough,  \eqref{e3.3} is satisfied with $c_1<1/2(\sqrt{2}-1)$. }

{\em In the case that  $x\mapsto Q^{1/2}B(x) \in L^{\infty}(H, \nu ; H)$ we need some restriction in order that the assumptions of Proposition \ref{p3.2} be satisfied. For instance, they are satisfied if $B= B_1+B_2$, with $ B_1 \in L^{\infty}(H, \nu ; H)$ and $ Q^{1/2}B_2 \in L^{\infty}(H, \nu ; H)$, 
$\|Q^{1/2}B_2\|_{\infty} \leq c_1 <1/2(\sqrt{2}+1)$. }
\end{Remark}

\subsection{Second perturbation}
 
In the case that $U\in C^2(H)$ we have also estimate \eqref{e2.19}, which is useful  when
\begin{equation}
\label{e3.6}
\langle D^2U(x)y, y\rangle \geq C(x)\|y\|^2,\quad x,y\in H, 
\end{equation}
and the function $C(x)$ is unbounded from above [if $C$ is bounded from above, \eqref{e2.19} does not add much information to \eqref{e2.17}].

\begin{Proposition}
\label{p3.4}
Let $U\in C^2(H)$ satisfy Hypothesis \ref{Hyp}. Assume moreover  that  \eqref{e3.6} holds for some unbounded $C(x)$ and that for every $\lambda>0$ and $f\in L^2(H, \nu)$ the weak solution $u$ to \eqref{e1.1} satisfies \eqref{e2.19}. Moreover, let $B:H\mapsto H$ be $\mu$-measurable and such that there exist $c_1$, $c_2$, $c_3>0$ with  $c_1^2+c_2^2< 1/8(\sqrt{2}+1)^2$, and for a.e. $x\in H$, we have
\begin{equation}
\label{e3.7}
|\langle B(x), y\rangle| \leq c_1 \|Q^{-1/2}y\| + c_2 \sqrt{C(x)} \|y\| + c_3\|y\|, \quad y\in Q^{1/2}(H). 
\end{equation}
Then the operator $K_1$  defined in \eqref{e3.1}  generates an analytic semigroup in $L^2(H, \nu)$. In particular, 
there exist  $\lambda_0\geq 0$, $C>0$ such that for every $\lambda>\lambda_0$ and for every $f\in L^2(H, \nu)$ the equation $\lambda v - K_1v =f$ has a unique solution $v\in D(K)$, and  
$$\|v\|_{D(K)} \leq C\|f\|_{L^2(H, \nu)}. $$
\end{Proposition}
\begin{proof}
We argue as in the proof of Proposition \ref{p3.2}. Here, besides estimates 
 \eqref{e3.4}  and \eqref{e3.5}, we also use  
\begin{equation}
\label{e3.8}
 \int_H \langle D^2U\,Du, Du\rangle  d\nu   \leq 4 \int_H(Ku)^2d\nu  , \quad u\in D(K), 
\end{equation}
which follows from \eqref{e2.19} taking $f=\lambda u - Ku$ and letting $\lambda \to 0$. 
By \eqref{e3.7} for each  $u\in D(K)$  we have
$$\int_H \langle B , Du \rangle ^2d\nu \leq  \int_H (c_1 \|Q^{-1/2}Du\| + c_2\sqrt{C(x)} \|Du\| + c_3\|Du\|)^2 d\nu. $$
Using the inequalities $(a+b+c)^2\leq a^2(2+\eps) + b^2(2+\eps) + c^2(1+2/\eps)$ for each $\eps\in (0, 1)$, and 
$$\int_H C(x)\|Du\|^2\,d\nu \leq \int_H  \langle D^2U\,Du, Du\rangle  d\nu   \leq 4 \int_H(Ku)^2d\nu $$
that follows from \eqref{e3.6} and  \eqref{e3.8}, 
we obtain, recalling \eqref{e3.4} and \eqref{e3.5},
$$\begin{array}{l}
\ds \int_H \langle B , Du \rangle ^2d\nu \leq 
\\
\\
\ds 
\leq c_1^2(2+\eps) \int_H  \|Q^{-1/2}Du\|^2d\nu + c_2^2 (2+\eps)\int_H C(x) \|Du\|^2d\nu + c_3^2\bigg(1+\frac{2}{\eps}\bigg)
\int_H  \|Du\|^2d\nu
\\
\\
\ds \leq 4(c_1^2 + c_2^2) (2+\eps) \int_H(Ku)^2d\nu + 
c_3^2\bigg(1+\frac{2}{\eps}\bigg)\bigg(4\lambda \int_Hu^2d\nu + 
\frac{4}{\lambda} \int_H (Ku)^2d\nu\bigg). 
\end{array}$$
As in the proof of Proposition   \ref{p3.2}, we may choose $\eps$ small and then $\lambda $ large, in such a way that for every $u\in D(K)$, we have 
$\int_H \langle B , Du \rangle ^2d\nu \leq a\int_H (Ku)^2d\nu  + b\int_H u^2d\nu  $ with $a<1/(\sqrt{2}+1)^2$, and the statement follows from Proposition \ref{p3.1}.   
\end{proof}

\begin{Remark}
{\em Assumption \eqref{e3.7} is satisfied if $B= B_1+B_2$, where 
$x\mapsto Q^{\alpha}B_1(x) \in L^{\infty}(H, \nu ; H)$ for some $\alpha \in [1/2)$ and there are $b <1/2(2+\sqrt{2})$, $c>0$ such that $\|B_2(x) \| \leq bC(x) + c$ for almost every $x\in H$. }
\end{Remark}

Theorem \ref{t2.10} allows to use Proposition \ref{p3.4} when $U\in C^2(H)$.  
In some specific examples the result of Proposition \ref{p3.4} holds when $U$ is not $C^2$, but belongs to a suitable Sobolev space. See Section 5.2.

\vspace{3mm}

We emphasize  that  the domain of the perturbed operator $K_1$ coincides with $D(K)$. 
Therefore, under the assumptions of Proposition \ref{p3.2}  for every $u\in D(K_1)$ we have
$$u\in W^{2,2}(H, \nu), \quad \int_H\|A^{-1/2}Du\|^2 d\nu <\infty, $$
and if  the assumptions of Proposition \ref{p3.4} hold, then for every $u\in D(K_1)$ we have also
$$\int_H\langle D^2U Du, Du\rangle \,d\nu <\infty .$$
An important feature of the semigroup generated by $K_1$ is positivity preserving. If $B\equiv 0$, that is $K_1=K$, Lemma \ref{l1.6} implies that $K$ satisfies the Beurling--Deny conditions that yield  positivity preserving (e.g., \cite[Sections 1.3, 1.4]{Davies}).

\begin{Proposition}
\label{p4.8}
Let the assumptions of Proposition \ref{p3.2} or of Proposition \ref{p3.4} hold, and let $\lambda_0$ be given by Proposition \ref{p3.2} or  \ref{p3.4}. 
Then for every $ \lambda >\lambda_0$ and $f\in L^2(H, \nu)$ such that $f(x)\geq 0$  a.e.,  $R(\lambda, K_1)f(x)\geq 0$ a.e. 
\end{Proposition}
\begin{proof}
Let us introduce the approximations
$$B_n(x) : = nR(n,A)B(x) \one_{\{x\in H:\;\|B(x)\|\leq n\}}, \quad n\in \N, \;x\in H,$$
that are $\mu$-measurable and   bounded in $H$. 

If the assumptions of Proposition \ref{p3.2}  hold, then each $B_n$ satisfies \eqref{e3.2} with the same constants $a$, $b$ of $B$. Indeed, since $\|nR(n,A)\|_{{\mathcal L}(H)}\leq 1$, then for every $x\in H$ and $y\in Q^{1/2}(H)$ we have
$$\begin{array}{l}
|\langle B_n(x), y\rangle | = |\langle B(x), nR(n, A)y\rangle | \one_{\{x\in H:\;\|B(x)\|\leq n\}}
\leq a \|Q^{-1/2}nR(n, A)y\| + b   \|nR(n, A) y\|
\\
\\
= a\|nR(n, A)Q^{-1/2}y\| + b   \|nR(n, A) y\|
\leq a\| Q^{-1/2}y\| + b \| y\|. 
\end{array}$$
Similarly, if the assumptions of Proposition \ref{p3.4}  hold, then $B_n$ satisfies
 \eqref{e3.7} with the same constants $c_1$, $c_2$, $c_3$ as $B$. Moreover 
 $B_{n}$ converges to  $B$ $\nu$-a.e., since 
$$B_{n}(x) - B(x) = nR(n, A)B(x) - B(x) \quad \mbox{\rm if} \;\|B(x)\| \leq n. $$

For each  $f\in L^2(H, \nu)$ we may approach  $R(\lambda, K_1)f$ by the solutions $u_{n}\in D(K)$ of problems
\begin{equation}
\label{e4.10}
\lambda u_{n} - Ku_{n} - \langle B_{n}(x), Du_{n}\rangle = f
\end{equation}
that still exist for  $\lambda >  \lambda_0$  since the functions $ B_{n}$ satisfy the assumptions of Proposition \ref{p3.1} (or, of Proposition \ref{p3.4}) with the same constants as $B$. By the proof of Propositions \ref{p3.2} and \ref{p3.4}, 
$u_{n}$ is obtained   as 
$R(\lambda, K)(I-T_{n})^{-1}$ where
$$T_{n}v = \langle B_{n}(\cdot), DR(\lambda, K)v\rangle , \quad v\in L^2(H, \nu), $$
and $(I-T_{n})^{-1}$ exists because $T$ is a contraction. We may use the principle  of contractions depending on a parameter, since
$$\|T_{n}v -Tv  \|_{ L^2(H, \nu)}^2 \leq  \int_H |\langle B - B_{n},  DR(\lambda, K)v\rangle |^2\,d\nu $$
that  vanishes as $n\to \infty$ by the dominated convergence theorem. Indeed, for $\nu$--almost every $x$ we have  $\lim_{n \to \infty} B_{n}(x) = B(x)$ and  
$$|\langle B_{n}(x), DR(\lambda, K)v(x)\rangle | \leq a\| Q^{-1/2}DR(\lambda, K)v(x)\| + b \| DR(\lambda, K)v(x)\|,$$
if the assumptions of Proposition \ref{p3.2} hold, and 
$$|\langle B_{n}(x), DR(\lambda, K)v(x)\rangle | \leq c_1\| Q^{-1/2}DR(\lambda, K)v(x)\| + c_2 \sqrt{C(x)}\| DR(\lambda, K)v(x)\| + c_3 \| DR(\lambda, K)v(x)\|, $$
if the assumptions of Proposition \ref{p3.4} hold. In both cases, the right-hand  sides belong to $L^2(H, \nu)$. 

It follows that for $\lambda > \lambda_0$ we have  $\lim_{n\to \infty} u_{n} = R(\lambda, K_1)f$, in $L^2(H, \nu)$. 
To finish the proof we show that if   $f\geq 0$ $\nu$-a.e., then $ u_{n}\geq 0$ $\nu$-a.e. This will yield the statement. 

Let us multiply both sides of  \eqref{e4.10} by  $u_{n}^-$, that belongs to  $W^{1,2}(H, \nu )$ by Lemma  \ref{l1.6}, and integrate over $H$. We get 
$$\lambda \int_H u_{n}\,u^-_{n}\, d \nu  +\frac{1}{2}\int_H \langle Du_{n}, Du^-_{n}\rangle \,  d\nu  - \int_H \langle B_{n} , Du_{n}  \rangle u^-_{n}\,  d\nu 
=  \int_H f\, u^-_{n}\,  d\nu ,$$
and recalling that  $u_{n}\,u_{n}^- = - (u_{n}^-)^2$, $\langle Du_{n}, Du_{n}^-\rangle = -\|Du_{n}^-\|^2$ by Lemma \ref{l1.6}, we obtain 
$$-\lambda \int_H (u_{n}^-)^2\,  d\nu  -\frac{1}{2} \int_H \|Du_{n}^-\|^2\,  d\nu 
- \int_H \langle B_{n} , Du_{n}  \rangle u^-_{n}\,  d\nu  \geq 0 .$$
Now we estimate
$$\bigg| \int_H \langle B_{n} , Du_{n}  \rangle u^-_{n}\,  d\nu \bigg| 
= \bigg| \int_{\{u_n \leq 0\}} \langle B_{n} , Du_{n}  \rangle u^-_{n}\,  d\nu \bigg| 
=  \bigg| \int_{H} \langle B_{n} , Du_{n}^-  \rangle u^-_{n}\,  d\nu \bigg| 
$$ 
$$\leq 
\|B_{n}\|_{\infty} \bigg(\int_H\|Du_{n}^-\|^2\,  d\nu \bigg)^{1/2} \bigg(\int_H (u_{n}^-)^2\,  d\nu \bigg)^{1/2} \leq \frac{1}{2} \int_H \|Du_{n}^-\|^2\,  d\nu  + 2\|B_{n}\|_{\infty}\int_H (u_{n}^-)^2\,  d\nu .$$
If $\lambda >C_n: = 2\|B_{n}\|_{\infty}$, we get 
$$-(\lambda -C_n)\|u_{n}^-\|^2_{L^2(H, \nu )} \geq 0$$
which implies  $u_{n}^-\equiv 0$, namely  $u_{n}\geq 0$ a.e. So, the resolvent of $K_{n}: = K +\langle B_{n}, D\cdot\rangle$ preserves positivity for $\lambda $ large, possibly depending on $n$. Since $K_{n}$ generates a $C_0$ semigroup, its resolvent preserves positivity for every $\lambda $ bigger than the type of the semigroup, in particular for every $\lambda >\lambda_0$. 
Then, 
$R(\lambda, K_1)$ preserves positivity for $\lambda >\lambda_0$. 
\end{proof}

Now we discuss the existence of an invariant  measure $\zeta(dx) = \rho(x)  \nu(dx) $ for the semigroup generated by $K_1$ in $L^2(H,\nu)$. An important step is the following proposition. 

\begin{Proposition}
\label{p4.9}
Let the assumptions of Proposition \ref{p3.2} or of Proposition \ref{p3.4} hold. Let in addition 
Hypothesis  \ref{Hyp1} hold. Then the kernel of $K_1^*$ $($the adjoint of $K_1$ in $L^2(H,\nu)$$)$ contains a nonnegative function $\rho \not\equiv 0$.  
\end{Proposition}
\begin{proof}
The  function $\one$ identically equal to $1$ belongs to the domain of $K_1$, and $  K_1\one = 0$. Then for any $\lambda >\lambda_0$, $\one$ is an eigenvector of $R(\lambda, K_1)$ with eigenvalue $1/\lambda$. Since $D(K_1)= D(K)$ is compactly embedded in $L^2(H, \nu)$ by Proposition \ref{p1.7},  then $R(\lambda, K_1)$ is a compact operator, and $1/\lambda $ is an eigenvalue of $R(\lambda, K_1)^* =  R(\lambda, K_1^*)$ too. Hence, 
$0$ is an eigenvalue of $K_1^*$, so that the kernel of $K_1^*$ contains nonzero elements. Note that since $R(\lambda, K_1)$ preserves positivity for large $\lambda$, then $R(\lambda, K_1^*)$ too preserves positivity for large $\lambda$, hence the semigroup $e^{tK_1^*}$ generated by $K_1^*$ preserves positivity for every $t>0$.

 Let us check that the kernel of $K_1^*$ is a lattice, that is, if   $\varphi \in \,$Ker$\,K_1^*$, then 
 $|\varphi|\in \,$Ker$\,K_1^*$.
Assume that $\varphi \in \,$Ker$\,K_1^*$. Then  $\varphi 
 = e^{tK_1^*}\varphi  $ for every $t>0$, and since $e^{tK_1^*}$ preserves positivity, then  
 $$ |\varphi (x)|=|e^{tK_1^*}\varphi (x)|\le (e^{tK_1^*}|\varphi |)(x),\quad  \nu -\;\mbox{a.e.}\; x\in H .
$$
We claim that for every $t>0$,
\begin{equation}
\label{lattice}
|\varphi (x)|= e^{tK_1^*}(|\varphi |)(x),\quad \nu -\;\mbox{a.e.}\; x\in H .
\end{equation}
Assume by  contradiction that there are $t>0$ and a Borel subset $I\subset H$ such that $\nu(I)>0$ and
$|\varphi (x)|< e^{tK_1^*}(|\varphi |)(x)$ for $x\in I$.
Then we have
$$
\int_{H}  |\varphi (x)|\nu (dx)<\int_{H}  (e^{tK_1^*}|\varphi|)(x)\nu (dx).
$$
On the other hand, since $\one \in $ Ker$\,K_1$, then $e^{tK_1^*}\one =\one$. Hence 
$$
\int_{H}e^{tK_1^*}|\varphi|\, d\nu  =\langle e^{tK_1^*}|\varphi| ,\one
 \rangle_{L^2(H,\nu)}=
\langle |\varphi|, e^{tK_1}\one   \rangle_{L^2(H,\nu)} =\int_{H}^{} |\varphi |d\nu ,
$$
which is  a contradiction. Then  \eqref{lattice} holds and it yields $|\varphi|\in $ Ker$\,K_1^*$. 
 \end{proof}

A realization of ${\mathcal K}_1$ in $L^2(H, \rho \nu)$ is m-dissipative, as the next proposition shows.

\begin{Proposition}
\label{p4.10}
Under the assumptions of Proposition \ref{p4.9}, let $\rho  $ be a  nonnegative function belonging to  {\em Ker}$\,K_1^*\setminus \{0\}$. 
Then 
the operator
$${\mathcal D}:= \{u\in D(K_1)\cap L^2(H, \rho \nu): \,K_1u\in L^2(H, \rho \nu)\}  \mapsto L^2(H, \rho \nu), \quad u\mapsto K_1u$$
is dissipative in $L^2(H, \rho \nu)$ and the range of $\lambda I - K_1: {\mathcal D} \mapsto L^2(H, \rho \nu)$ is dense in $L^2(H, \rho \nu)$ for $\lambda >0$. Then  its closure  $\widetilde{K}_1$ generates a contraction semigroup $\widetilde{T}_1(t)$ in $L^2(H, \rho \nu)$, and the measure $\rho \nu$ is invariant for $\widetilde{T}_1(t)$.  
\end{Proposition}
\begin{proof} As a first step we prove dissipativity, through estimates on $R(\lambda, K_1)$. 

We remark that Lemma \ref{l1.2} holds for the measure $\rho\nu$ as well, with the same proof. In particular, $C_b(H)$ is dense in $L^1(H, \rho \nu)$. 

Let $\lambda >\lambda_0$ and let  $f\in C_b(H)$. 
Set $u=R(\lambda, K_1)f$. We recall that, since $\rho \in D(K_1^*)$ and 
$K_1^*\rho =0$, then for every $u\in D(K_1)$ we have $\int_H K_1u\,\rho \, d\nu = \int_H u \,K_1^*\rho \,d\nu =0$. 
So, multiplying  both sides of $\lambda u - K_1u = f$ by $\rho$ and integrating we obtain 
$$\int_H \lambda u \,\rho \, d\nu = \int_H f \,\rho \, d\nu .$$
If $f$ has nonnegative values $\nu$-a.e., by Proposition \ref{p4.8} $u$ has nonnegative values $\nu$-a.e., and the above equality  implies
\begin{equation}
\label{e4.11}
\|u\|_{L^1(H, \rho \nu)} \leq \frac{1}{\lambda} \|f\|_{L^1(H, \rho \nu)}.
\end{equation}
In general, we split $f $ as  $f = f^+ - f^-$. Since $u = R(\lambda, K_1)  f^+ - R(\lambda, K_1)f^- = u^+ -u^-$, 
\eqref{e4.11} follows for every $f\in C_b(H)$.  Since $C_b(H)$ is dense in $L^1(H, \rho \nu)$, the resolvent $R(\lambda, K_1)$ may be extended to a bounded operator (still denoted by  $R(\lambda, K_1)$) to $L^1(H, \rho \nu)$, and
\begin{equation}
\label{e4.12}
\|R(\lambda, K_1)f\|_{L^1(H, \rho \nu)} \leq \frac{1}{\lambda} \|f\|_{L^1(H, \rho \nu)}, \quad f\in L^1(H, \rho \nu).
\end{equation}

Let now $f\in L^{\infty}(H, \rho \nu)$. 
$f$ is in fact an equivalence class of functions, that contains a Borel bounded element. Indeed, for each element $\varphi \in f$, setting $\widetilde{f}(x) = \varphi(x)$ if $|\varphi(x)|\leq \|f\|_{L^{\infty}(H, \rho \nu)}$, $\widetilde{f}(x) = 0$ if $|\varphi(x)|>\|f\|_{L^{\infty}(H, \rho \nu)}$, the function $\widetilde{f}$ is Borel and bounded, and $\|f\|_{L^{\infty}(H, \rho \nu)} = \sup_{x\in H}|\widetilde{f}(x)|$. 
 
Let us go back to the resolvent equation, $\lambda u - K_1u = \widetilde{f}$. Since $\widetilde{f}$ is Borel and bounded, it can be seen  as an element of $L^{\infty}(H,  \nu)$, identifying it with its equivalence class\footnote{
Note that  $\rho$ may vanish on some set with positive measure, so that $f$ does not belong necessarily to $L^{\infty}(H,  \nu)$, and even it does, its $L^{\infty}(H,  \nu)$ norm may be bigger than its $L^{\infty}(H, \rho \nu)$ norm.}.
Moreover, $\|\widetilde{f}\|_{L^{\infty}(H,  \nu)} = \sup_{x\in H}|\widetilde{f}(x)| = \|\widetilde{f}\|_{L^{\infty}(H, \rho \nu)} $.

Since $ \sup |\widetilde{f}| - \widetilde{f}(x) \geq 0 $ for every $x$, still by Proposition \ref{p4.8} we have 
$R(\lambda, K_1)( \sup |\widetilde{f}| - \widetilde{f}) =  \sup |\widetilde{f}| /\lambda - u\geq 0$, $\nu$-a.e. Similarly, since $\widetilde{f}(x) +  \sup |\widetilde{f}|  \geq 0 $ for every $x$, then $ u + \sup |\widetilde{f}| /\lambda \geq 0$, $\nu$-a.e. 
So, we get an $L^{\infty}$  estimate, $\|u\|_{L^{\infty}(H,  \nu)} \leq  \sup |\widetilde{f}| /\lambda $.
Hence 
\begin{equation}
\label{e4.13}
\|R(\lambda, K_1)f\|_{L^{\infty}(H, \rho \nu)} \leq \|R(\lambda, K_1)\widetilde{f}\|_{L^{\infty}(H,  \nu)}\leq 
\frac{1}{\lambda}\|f\|_{L^{\infty}(H, \rho \nu)}, \quad f\in  L^{\infty}(H, \rho \nu).
\end{equation}
By interpolation, $R(\lambda, K_1)$ may be extended to $L^{2}(H, \rho  \nu)$ [and, in fact, to all spaces $L^{p}(H, \rho  \nu)$], in such a way that the norm of the extension does not exceed $1/\lambda$. In particular,
\begin{equation}
\label{e4.14}
\|R(\lambda, K_1)f\|_{L^{2}(H, \rho \nu)} \leq \frac{1}{\lambda} \|f\|_{L^{2}(H, \rho  \nu)}, \quad f\in L^{2}(H, \rho \nu)
\cap L^2(H, \nu).
\end{equation}
Let now $u\in {\mathcal D}$. For $\lambda >\lambda_0$ estimate \eqref{e4.14} gives
$$\lambda \|u\|_{L^{2}(H, \rho \nu)} \leq   \|\lambda u - K_1 u\|_{L^{2}(H, \rho  \nu)}$$
and squaring the norms of both sides, we obtain
$$  \langle u, K_1u\rangle_{L^{2}(H, \rho  \nu)}\leq \frac{1}{2\lambda}\| K_1 u\|_{L^{2}(H, \rho  \nu)}^2. $$
Letting $\lambda \to \infty$ yields $  \langle u, K_1u\rangle_{L^{2}(H, \rho  \nu)}\leq 0$, namely the restriction of $K_1$ to ${\mathcal D}$ is dissipative in $L^{2}(H, \rho  \nu)$. 

We remark that  ${\mathcal D}$ is dense in $L^2(H, \rho \nu)$ since it contains ${\mathcal F}{\mathcal C}^{\infty}_b(H)$ which is dense by the extension of Lemma \ref{l1.2} to $L^2(H, \rho \nu)$. Moreover  $(\lambda I -K_1)({\mathcal D})$ is dense for $\lambda >\omega_0$, since it contains ${\mathcal F}{\mathcal C}^{\infty}_b(H)$. 
Indeed, if $f\in {\mathcal F}{\mathcal C}^{\infty}_b(H)$, then  $u= R(\lambda, K_1)f $ belongs to ${\mathcal D}$ and $\lambda u - K_1u =f$. 

Let us denote by $\widetilde{K}_1:D(\widetilde{K}_1)\mapsto L^2(H, \rho \nu)$ the closure of $K_1: {\mathcal D}   \mapsto L^2(H, \rho \nu)$. 
By the Lumer--Phillips Theorem, $\widetilde{K}_1$ generates a strongly continuous contraction semigroup in $L^2(H, \rho \nu)$, and ${\mathcal D}$ is a core for $\widetilde{K}_1$. 
So, for every $\varphi\in D(\widetilde{K}_1)$ there is a sequence of functions $\varphi_n\in {\mathcal D}$ such that $\varphi_n\to \varphi $ and $K_1 \varphi_n \to \widetilde{K}_1\varphi$ in $L^2(H, \rho \nu)$. For every $n$ we have 
$$\int_H K_1 \varphi_n \,\rho \,d\nu = \int_H  \varphi_n \,K_1^*\rho \,d\nu =0$$
and letting $n\to \infty$ we obtain $\int_H \widetilde{K}_1 \varphi  \,\rho \,d\nu =0$. This proves the last statement. 
\end{proof}

%%%%%%%%%%%%%%%%%%%%%%%%%%%%%%%%%%%%%%%%%%%%%%%%%%%%%%%
\section{Kolmogorov equations of stochastic  reaction-diffusion equations.}
%%%%%%%%%%%%%%%%%%%%%%%%%%%%%%%%%%%%%%%%%%%%%%%%%%%%%%%

Let $H=L^2((0,1), d\xi)$, and let $A$ be the realization of the second order derivative with Dirichlet boundary condition, that is $D(A)=W^{2,2}((0, \pi), d\xi) \cap W^{1,2}_0((0, \pi), d\xi)$, $Ax = x''$. 

We consider  the Gaussian measure $\mu $ in $H$ with mean $0$ and covariance $Q := -\frac12\,A^{-1}$. 
A canonical orthonormal basis of $H$ consists of the functions $e_k(\xi) := \sqrt{2}\sin(k \pi \xi)$, $k\in \N$, that are eigenfunctions of $Q$ with eigenvalues $\lambda_k := 1/(2k^2 \pi^2)$.

Let $\Phi:\R\mapsto \R$ be any convex lowerly bounded function, with (at most) polynomial growth at infinity, say 
\begin{equation}
\label{e4.1}
 |\Phi(t)| \leq C(1+ |t|^{p_1}), \quad  t\in \R, 
 \end{equation}
for some $C>0$, $p_1 \geq 2$. We set
\begin{equation}
\label{e4.2}
U(x)= \left\{ \begin{array}{lll}  & \ds \int_0^{1} \Phi(x(\xi))d\xi, & x\in L^{p_1}(0,1), 
\\
\\
 & +\infty , & x \notin L^{p_1}(0,1). 
\end{array}\right. 
\end{equation}

\noindent  Section  5.1  is devoted to  check that $U$ satisfies
Hypotheses \ref{Hyp} and \ref{Hyp1}, so that we can apply Theorem \ref{t2.7} to obtain regularity results for the solution $u$ to \eqref{e1.1}.
 Then in Section 5.2 we show that under an additional assumption $u$ fulfills \eqref{e2.19} too.

\subsection{Checking Hypotheses \ref{Hyp} and \ref{Hyp1}}

We first note that $U$ is finite $\mu$--a.e., thanks to the next lemma.  Its statement should be well known; however, we write down a simple proof for the reader's convenience. 

\begin{Lemma}
\label{l4.1}
For every $p\geq 2$ we have
\begin{equation}
\label{e4.3}
\int_H \int_0^1 | x(\xi)|^{p} d\xi\, d\mu < \infty ,
\end{equation}
and hence $\mu(L^p(0,1)) =1$. Moreover, $x\mapsto \|x\|_{L^p(0,1)} \in L^q(H, \mu)$ for every $q\geq 1$. 
\end{Lemma}
\begin{proof}
Let  $P_n$ be the orthogonal projection on the subspace spanned by $e_1, \ldots, e_n$. 
For every $\xi\in (0, 1)$ and $m < n \in \N$, 
the function  $x\mapsto P_nx (\xi)- P_mx(\xi)$ is a Gaussian random variable $N_{0, \sum_{k=m+1}^n\lambda_ke_k(\xi)^2}$. Then, for $p\geq 1$, 
$$\int_H |P_nx (\xi) - P_mx(\xi)|^{p} d\mu = \int_{\R}|\eta| ^{p}N_{0, \sum_{k=m+1}^n \lambda_k e_k(\xi)^2}(d\eta) $$
$$= c_p 
 (\sum_{k=m+1}^n \lambda_k e_k(\xi)^2)^{p/2} \leq \tilde{c}_p  (\sum_{k=m+1}^n \lambda_k )^{p/2}, $$
with $\tilde{c}_p= 2^{p/2}c_p$,  so that 
$$\int_H \int_0^1 |P_nx (\xi) - P_mx(\xi)|^{p} d\xi\, d\mu = \int_0^1 \int_H |P_nx (\xi) - P_mx(\xi)|^{p}  d\mu \, d\xi \leq \tilde{c}_p  (\sum_{k=m+1}^n \lambda_k )^{p/2}. $$
This implies  that the sequence $(x,\xi)\mapsto P_nx(\xi)$ converges in $L^p(H\times (0,1), \mu\times d\xi)$ to a limit function $u$ that belongs to $L^p(H\times (0,1), \mu\times d\xi)$ for every $p$. Let us show that $u(x,\xi) = x(\xi)$ taking $p=2$: indeed, $\int_0^1 |P_nx (\xi) -  x(\xi)|^{2} d\xi $ vanishes for every $x\in H$ as $n\to \infty$, and it is bounded by $ \|x\|^2$ which belongs to $L^1(H, \mu)$, so that by the dominated convergence theorem, $\int_H \int_0^1 |P_nx (\xi) -  x(\xi)|^{2} d\xi\, d\mu$ vanishes as $n\to \infty$. Then  $u(x,\xi) = x(\xi)$, and \eqref{e4.3} follows. 
It implies that $\mu(L^p(H, \mu)) =1$ for every $p\geq 2$ and that $x\mapsto \|x\|_{L^p(0,1)} \in L^p(H, \mu)$. For $q>p$ and $x\in L^q(0,1)$ the H\"older inequality yields $\|x\|_{L^p(0,1)} \leq \|x\|_{L^q(0,1)}$ so that $x\mapsto \|x\|_{L^p(0,1)}   \in L^q(H, \mu)$. 
\end{proof}

The function $U$ defined by \eqref{e4.2}  is convex and bounded from below because $\Phi$ is. Using the Fatou Lemma, it is easily seen to be lowerly semicontinuous. By assumption \eqref{e4.1} and Lemma \ref{l4.1},  $U\in L^p(H, \mu)$ for every $p\geq 1$, and the measures $\mu$ and $\nu = e^{-2U}\mu/ \int_He^{-2U}d\mu$ are equivalent. For $U$ belong to some Sobolev space it is sufficient that also $\Phi'$ has at most polynomial growth, as the next proposition shows.

\begin{Proposition}
\label{p4.2}
Let $\Phi:\R\mapsto \R$ be any $C^1$ convex lowerly bounded function such that 
\begin{equation}
\label{e4.4}
 |\Phi'(t)| \leq C(1+ |t|^{p_2}), \quad  t\in \R, 
 \end{equation}
for some $C>0$, $p_2\geq 1$. Then the function $U$ defined in \eqref{e4.2} belongs to $W^{1,p}_0(H, \mu)$ for every $p\geq 1$, and $DU(x) = \Phi '\circ x$ for a.e. $x\in H$ [namely, for each  $x\in L^{2p_2}(0,1)$]. 
\end{Proposition}
\begin{proof}
By \eqref{e4.4}, $\Phi$ satisfies \eqref{e4.1} with $p_1=p_2+1$, so that $U\in L^p(H, \mu)$ for every $p$ by Lemma \ref{l4.1}. To prove that  $U  \in W^{1,p}_{0}(H,\mu)$ we
shall approach $U$ by its Moreau--Yosida approximations $U_{\alpha}$ defined in \eqref{e1.8}. 
Each $U_{\alpha}$ is continuously differentiable and  $DU_{\alpha}$ is Lipschitz continuous, hence $U_{\alpha} \in W^{1,p}_{0}(H,\mu)$ for every $p$. This can be easily proved arguing as in the case $p=2$ of \cite[Proposition 10.11]{Beppe}. 

Since $U_{\alpha}(x)$ converges monotonically to $U(x)$ at each $x$ such that $U(x) <\infty$, by Lemma \ref{l4.1} $U_{\alpha}$ converges to $U$, $\mu$- a.e. Since 
$$\inf U\leq U_{\alpha}(x) \leq U(x) \leq C(1+ \int_0^1|x(\xi)|^{p_1} d\xi)\leq C(1+ (\int_0^1|x(\xi)|^{p_1p} d\xi)^{1/p}),$$
by Lemma \ref{l4.1} and the dominated convergence theorem,  $U_{\alpha}\to U$ in $L^p(H, \mu)$. 

Let $x\in L^{2p_2}(0, 1)$. Then the  subdifferential $\partial U(x)$ is not empty. Indeed, since $\Phi$ is convex,  for each $y\in H$ we have 
\begin{equation}
\label{e4.6}
U(y) - U(x) = \int_0^{\pi} [\Phi (x(\xi)) - \Phi (y(\xi)) ]\,d\xi \geq \int_0^{\pi} \Phi '(x(\xi))(x(\xi) - y(\xi))d\xi,  
\end{equation}
which implies that the function $\Phi '\circ x \in H$ belongs to $\partial U(x)$. In fact, $\Phi '\circ x \in H$ is the unique element of $\partial U(x)$; see, for example, \cite[Prop. 2.5]{Barbu}. 
By  Lemma \ref{l4.1}, $x\mapsto \|\Phi '\circ x\| \in L^p(H, \mu)$, and again by the dominated convergence theorem 
$\int_H  \|DU_{\alpha}(x) - \Phi '\circ x\|^pd\mu \to 0$ as $\alpha \to 0$, which shows that $U\in W^{1,p}_{0}(H,\mu)$ and $DU(x)= \Phi '\circ x$, $\mu$-a.e. \end{proof}

If the assumptions of Proposition \ref{p4.2} hold, then
 $U$ satisfies Hypothesis \ref{Hyp} and \ref{Hyp1}, and consequently  the results of Theorem \ref{t2.7} and of Propositions \ref{p4.9} and \ref{p4.10} hold.

\subsection{Further estimates of $Du$} 
 
 We are going to show that for every $\lambda >0$ and $f\in L^2(H, \nu)$, the solution of \eqref{e1.1} satisfies estimate \eqref{e2.19}  as well, under reasonable additional assumptions on $\Phi$. We use  the following preliminary result.

\begin{Proposition}
\label{p4.3}
Let  $g\in C^2( \R)$ be such that 
\begin{equation}
\label{e4.7}
|g''(t)|\leq C(1 + |t|^{m}), \quad t\in \R .
\end{equation}
for some $C>0$, $m\geq 1$. Then the function  $F(x):= g\circ x$ belongs to $W^{1,q}_{1/2}(H, \mu; H)$ for all  $q>1$. If in addition $g_{\alpha}:\R\mapsto \R$ are  $C^2$ functions fulfilling  \eqref{e4.7}  with  constant $C$ independent of $\alpha >0$ and  $g_{\alpha}$, $g_{\alpha}'$ pointwise converge to $g$, $g'$ respectively as $\alpha \to 0^+$, then
$F_\alpha(x):=g_{\alpha}\circ x $ converges to $F$ in $W^{1,q}_{1/2}(H, \mu; H)$ as $\alpha\to 0^+$ for all  $q>1$. 
 \end{Proposition}
\begin{proof}
As first step we show that  for each  $x\in L^{2m}(0,1)$ (hence,  $\mu$-a.e.), $F$ is differentiable in any direction $h\in Q^{1/2}(H) = H^1_0(0,1)$ and that  $\frac{\partial F(x)}{\partial h} = g'\circ x \cdot h $. 
We have in fact for all  $h\in H^1_0(0,1)$, $\xi \in (0, 1)$ and all  $0< |t|\leq 1$, 
$$\begin{array}{l}
\ds \bigg| \frac{g(x + th)(\xi) - g(x(\xi))}{t} - g'(x(\xi))h(\xi)\bigg| 
  = \bigg|  \int_0^1 [g'(x(\xi) + t\sigma h(\xi))- g'(x(\xi))]h(\xi) \,d\sigma \bigg| 
\\
\\
= \ds \bigg|  \int_0^1  \int_0^1g''(x(\xi) + t\sigma \eta h(\xi)) t \sigma h(\xi)^2\,d\eta  \,d\sigma \bigg| 
\leq t\|h\|_{\infty}^2 C(1 + 2^{m-1}|(|x(\xi)|^{m}+\|h\|_{\infty}^{m})).
\end{array}$$
Now, taking the square and integrating over $(0,1)$, 
yields
$$\left\| \frac{F(x+th) - F(x)}{t} - g'\circ x \cdot h\right\|_H \leq t C(h)\left(1+ \|x\|_{L^{2m}}^{m}\right).$$
This implies that for each  $x\in L^{2m}(0,1)$, $F$  is differentiable at $x$ in any direction  $h\in H^1_0(0,1)$ and that
$$\frac{\partial F(x)}{\partial h} = g'\circ x \cdot h .$$
Let us notice  that $F$, $\partial F/\partial h$  belong to $L^q(H, \mu; H)$ for every $q\geq 1$. Indeed, \eqref{e4.7} implies that $|g(t)|\leq M(1+|t|^{m+2})$, $|g'(t)|\leq M(1+|t|^{m+1})$ for every $t\in \R$ and for some $M>0$, so that 
$|F(x(\xi))| \leq M(1+|x(\xi)|^{m+2})$, $|\partial F(x)/\partial h (\xi)| \leq M(1+|x(\xi)|^{m+1})\|h\|_{\infty}$
and then
$$\|F(x)\|^2_H \leq \int_0^1M^2 (1+|x(\xi)|^{m+2})^2d\xi, \quad \bigg\| \frac{\partial F(x)}{\partial h} (x)\bigg\|_H^2 \leq \|h\|_{\infty}^2\int_0^1M^2 (1+|x(\xi)|^{m+1})^2d\xi ,$$
and the right-hand  sides belong to $L^q(H, \mu)$ for every $q$. 
It follows from \cite[Section 5.2]{Boga} that  $F$ belongs to $G^{q, 1}(H, \mu; H)$ (i.e.,   $F$ belongs to $L^{q}(H, \mu; H)$, it is weakly differentiable in all directions of the Cameron--Martin space $H^1_0(0,1)$ and any weak derivative $\frac{\partial F(x)}{\partial h}$ with $h\in H^1_0(0,1)$ can be expressed as $\Psi(x)h$, where $\Psi \in L^q(H, \mu; {\mathcal L}(H^1_0(0,1), H))$ is such that $\partial F(x)/\partial h = \Psi(x)(h)$). To show that   $F\in W^{1,q}_{1/2}(H, \mu; H)$  we have still to check that (\cite[Proposition 5.4.6, Corollary 5.4.7]{Boga})
$$\int_H \bigg(\sum_{h, k\in \N}\lambda_h\lambda_k  \langle \partial F(x)/\partial e_h, e_k\rangle ^2\bigg)^{q/2}d\mu <\infty . $$
This is because a canonical orthonormal basis of   $H^1_0(0,1)$ is just the set $\{\sqrt{\lambda_k}e_k:\;k\in \N\}$. 
Recalling that $\|e_k\|_{\infty} = \sqrt{2}$ for every $k$,  we get
$$|\langle \partial F(x)/\partial e_h, e_k\rangle | = \bigg| \int_0^1 g'(x(\xi))e_h(\xi)e_k(\xi)d\xi \bigg|
\leq 2M\int_0^1(1+|x(\xi)|^{m+1})d\xi = 2M(1+ \|x\|_{L^{ m+1}}^{m+1})$$
for each  $h$, $k\in \N$,  which implies
$$\int_H \bigg(\sum_{h, k\in \N}\lambda_h\lambda_k  \langle \partial F(x)/\partial e_h, e_k\rangle ^2\bigg)^{q/2}d\mu \leq 2M \int_H ({\rm Tr}\,Q)^q\|x\|_{L^{ m+1}}^{q(m+1)}d\mu <\infty , $$
so that $F\in W^{1,q}_{1/2}(H, \mu; H)$.

Now we can   show that  $F_\alpha\to F$ as $\alpha\to 0$. In fact, since   \eqref{e4.7} is fulfilled with constant independent of $\alpha$, there is $M_1>0$ independent of $\alpha$ such that 
$$|g_{\alpha}(t)|\leq M_1(1+|t|^{m +2}), \quad |g_{\alpha}'(t)|\leq M_1(1+|t|^{m+1}), \quad t\in \R .$$
Concerning  the  convergence of  $g_{\alpha}\circ x$ to  $g\circ x$ in $L^q(H, \mu;H)$ we have
$$\begin{array}{l}
\ds \int_H \|g_{\alpha}\circ x - g\circ x\|_H^q d\mu = \int_H \bigg( \int_0^1 |g_{\alpha}(x(\xi))- g(x(\xi))| ^2d\xi \bigg)^{q/2}d\mu 
\\
\\
\ds \leq  \int_H \int_0^1 |g_{\alpha}(x(\xi))- g(x(\xi))| ^{q}d\xi \, d\mu \end{array}, $$
and   the last integral goes to $0$ as  $\alpha\to 0$ by the dominated convergence theorem. Therefore $F_{\alpha}(x) =  g_{\alpha}\circ x$ converges to $F$ in $L^q(H,\mu; H)$. 
Concerning the convergence  in $W^{1,q}_{1/2}(H, \mu;H)$ we have
$$\begin{array}{l}\ds \int_H \bigg(\sum_{h, k\in \N}\lambda_h\lambda_k  \langle  \partial (g_{\alpha}\circ x)/\partial e_h - \partial (g\circ x)/\partial e_h , e_k\rangle ^2\bigg)^{q/2}d\mu \\
\\
\ds = \int_H \bigg(\sum_{h, k\in \N}\lambda_h\lambda_k  \bigg(\int_0^1(g_{\alpha}'(x(\xi))- g'(x(\xi)))e_h(\xi) e_k(\xi)d\xi\bigg)^2\bigg)^{q/2}d\mu \end{array}$$
$$\begin{array}{l}\ 
\ds \leq C_q \int_H \bigg(\sum_{h, k\in \N}\lambda_h\lambda_k \int_0^1 |g_{\alpha}'(x(\xi))- g'(x(\xi))|^2d\xi\bigg)^{q/2}d\mu
\\
\\
\ds \leq C_q ({\rm Tr}\,Q)^q \int_H \int_0^1 |g_{\alpha}'(x(\xi))- g'(x(\xi))|^q d\xi\,d\mu ,
 \end{array}$$
and   the last integral vanishes as $\alpha\to 0$ again by   the dominated convergence theorem.    \end{proof}

We shall use Proposition \ref{p4.3} to prove that the Moreau--Yosida approximations   $U_{\alpha}$ converge  to  $U$  in $W^{2,q}_{1/2}(H, \mu )$ for every $q$ [for the moment, we only know  convergence in $W^{1,q} (H, \mu )$].

\begin{Proposition}
\label{p4.4}
Let $\Phi:\R\mapsto \R$ be any $C^3$ convex lowerly bounded function such that 
\begin{equation}
\label{e4.8}
|\Phi'''(t)|\leq C(1 + |t|^{m}), \quad t\in \R , 
\end{equation}
for some $C$, $m>0$. 
Then  $U\in W^{2,q}_{1/2}(H, \mu )$ for all $q>1$, and we have 
$$\lim_{\alpha \to 0}U_{\alpha} = U\quad\mbox{ in}\;\;W^{2,q}_{1/2}(H, \mu ),\;\forall\;q>1.$$
 \end{Proposition}
\begin{proof}
Let us apply Proposition \ref{p4.3} to $F(x)  = DU(x) = g\circ x $ with $g= \Phi' $.  Since $g''$ has   polynomial growth, $F\in W^{2,q}_{1/2}(H, \mu; H)$ for all $q$, so that $U\in W^{2,q}_{1/2}(H, \mu)$ for all  $q$. Moreover 
$DU_{\alpha}(x) = D_0U(y_{\alpha})$, where $y_{\alpha}$ is the solution of 
$$y_{\alpha} + \alpha D_0U(y_{\alpha}) = x,$$ 
that is
$$y_{\alpha} + \alpha \Phi'(y_{\alpha}) = x.$$
Therefore
$$y_{\alpha}(\xi) = (I+ \alpha \Phi ')^{-1}(x(\xi)), \quad 0<\xi < 1, $$
and so
$$DU_{\alpha}(x) = \Phi' \circ (I+ \alpha \Phi ')^{-1}\circ x . $$
Setting  $g_{\alpha}(t) = \Phi' \circ (I+ \alpha \Phi ')^{-1}(t)$, we see that $g_{\alpha}$  converges pointwise to $g=\Phi'$, and 
$$g_{\alpha}'= \frac{\Phi''\circ (I+ \alpha \Phi ')^{-1}}{(1+\alpha \Phi''\circ (I+ \alpha \Phi ')^{-1})}$$ converges pointwise to  $g'= \Phi''$. 

Moreover we notice that there exists $M>0$,  independent of $\alpha \in (0, 1)$ such that $|(I+ \alpha \Phi ')^{-1}(t)|\leq M + |t|$ for all $t\in \R$.   \eqref{e4.8} implies that $\Phi'$ and $\Phi ''$ have polynomial growth as well; in particular $|\Phi '(t)| \leq c_1(1+|t|^{m+2})$, so that   $|g_{\alpha}(t) | \leq c_1(1+ (M + |t|)^{m+2})$. A similar estimate with $m+1$ instead of $m+2$ holds also  for $|g_{\alpha}'(t) |$. By the second part of Proposition \ref{p4.3},   $DU_{\alpha}$ converges to  $DU$ in $W^{1,q}_{1/2}(H, \mu; H)$ as $\alpha\to 0$, thereby
$U_{\alpha}$ converges to $U$ in $W^{2,q}_{1/2}(H, \mu )$. 
\end{proof} 

As a final step,  we can show that the solution to \eqref{e1.1} satisfies \eqref{e2.19}  under the assumptions of Proposition \ref{p4.4}.

\begin{Proposition}
\label{p4.5}
Let  $U$ be defined by \eqref{e4.2} with  $\Phi:\R\mapsto \R$ convex, bounded from below, of 
class $C^3$ and satisfying \eqref{e4.8}.  Then for every $\lambda >0$ and $f\in L^2(H, \nu)$ the weak solution $u$ of \eqref{e1.1} satisfies \eqref{e2.19}. 
 \end{Proposition}
\begin{proof}
It is sufficient to prove the statement for  $f\in C_b(H)$, which is dense in $L^2(H, \nu)$. By Lemma \ref{l2.8} there is a sequence   $(\alpha_n)\to 0$   such that $ u_{\alpha_n} \to u$ in $W^{1,2}(H, \nu)$. Then  
$Du_{\alpha_n} \to Du$ in $L^2(H, \nu; H)$ so that (up to a subsequence) $Du_{\alpha_n}(x) \to Du(x)$  for almost all $x$. 
By Proposition \ref{p4.4}, $U_{\alpha_n}$ converges to $U$ in $W^{2,2}_{1/2}(H, \mu)$, thereby for all  fixed  $h, k\in \N$ we have $D_{hk}U_{\alpha_n}\to D_{hk}U$ in $L^2(H, \mu)$. 
Let us fix $N\in \N$. Possibly choosing a further subsequence, we have  $D_{hk}U_{\alpha_n}\to D_{hk}U$ pointwise a.e.  for all  $h, k\leq N$. 
Therefore for $\mu$- a.e. $x\in H$ we have
$$\lim_{n\to \infty} \sum_{h,k=1}^N D_{hk}U_{\alpha_n}(x) D_hu_{\alpha_n}(x) D_ku_{\alpha_n}(x)  e^{-2U_{\alpha_n}(x)}=
\sum_{h,k=1}^N  D_{hk}U (x) D_hu (x) D_k(x)   e^{-2U (x)}$$
and by Fatou's lemma, 
$$\begin{array}{l}
\ds \int_H \sum_{h,k=1}^N D_{hk}U (x) D_hu (x)  D_ku(x)\,   d\nu 
 = \int_H \sum_{h,k=1}^N D_{hk}U (x) D_hu (x)  D_ku(x)\, e^{-2U (x)} d\mu 
 \\
 \\
 \ds \leq \liminf_{n\to \infty} \int_H 
\sum_{h,k=1}^N D_{hk}U_{\alpha_n}(x) D_hu_{\alpha_n}(x)  D_ku_{\alpha_n}(x)\, e^{-2U_{\alpha_n}(x)}d\mu
\\
\\
\ds \leq 4 \liminf_{n\to \infty} \int_H f^2 \, e^{-2U_{\alpha_n}}d\mu = 4  \int_H f^2 \, d\nu .
\end{array}$$
Now by Theorem \ref{t2.7} we know that $x\mapsto \|Du(x)\|_{H^1_0(0,1)} = \|Q^{-1/2}Du(x)\|_H/\sqrt{2}$
$\in L^2(H, \mu)$, therefore for almost any  $x\in H$, $Du(x)\in H^1_0(0,1)$, whereas by  Proposition \ref{p4.4} it follows that  $x\mapsto  \sum_{h,k=1}^{\infty}\lambda_h\lambda_k (D_{hk}U(x))^2$ belongs to  $L^1(H, \mu)$, that is      $x\mapsto \|D^2U(x)\|_{{\mathcal L}_2(H^1_0(0,1))} \in L^2(H, \mu)$. Therefore  for almost   $x\in H$, $D^2U(x)\in {\mathcal L}_2(H^1_0(0,1))$.  It follows that  for almost any  $x\in H$ the sequence
$\sum_{h,k=1}^{N}D_{hk}U(x)D_ku(x)D_ku(x)
 $ converges to $\sum_{h,k=1}^{\infty}D_{hk}U(x)D_ku(x)D_ku(x)$. 
Using once again   Fatou's lemma we can  conclude 
that
$$\begin{array}{l}
\ds \int_H \sum_{h,k=1}^{\infty} D_{hk}U (x) D_hu (x)  D_ku(x)\,   d\nu 
=  \int_H  \lim_{N\to \infty} \sum_{h,k=1}^N D_{hk}U (x) D_hu (x)  D_ku(x)\,   d\nu 
\\
\\
\ds \leq \liminf_{N\to \infty} 
 \int_H   \sum_{h,k=1}^N D_{hk}U (x) D_hu (x)  D_ku(x)\,   d\nu 
\leq 4  \int_H f^2 \, d\nu .\end{array}$$
\end{proof}

Then we can apply all the results of Sections 3 and 4. In particular, we have the following theorem.

\begin{Theorem}
\label{t4.6}
 Let $\Phi:\R\mapsto \R$ be any convex $C^1$ lowerly bounded function satisfying \eqref{e4.4}, and let $U$ be defined by \eqref{e4.2}.   Then for every $\lambda >0$ and $f\in L^2(H, \nu)$ the weak solution $u$   to \eqref{e1.1}  belongs to $W^{2,2}(H, \nu)$ $\cap $ $W^{1,2}_{-1/2}(H, \nu)$, and it satisfies \eqref{e2.17}, \eqref{e2.18}. 
If in addition $\Phi$ is $C^3$ and  satisfies \eqref{e4.8}, then  
 $u$ satisfies   \eqref{e2.19} as well. 
\end{Theorem}

With our choice of $U$,  the stochastic differential equation \eqref{e1.3}
in $H$  reads as 
\begin{equation}
\label{react-diff}
dX=(AX-\Phi'(X))dt+dW(t),\quad X(0)=x,  
\end{equation}
and hence it is a  reaction-diffusion SPDE, whose Kolmogorov operator is just $\mathcal K$. 
 As in Section \ref{Weak = strong}, $W(t)$ is any $H$-valued cylindrical Wiener process defined in a probability space $(\Omega, {\mathcal F}, \P)$. 
 The connection between \eqref{react-diff} and  \eqref{e1.1} is stated in the next proposition. The definition of mild solution to \eqref{react-diff} is the same as in the case of Lipschitz continuous $DU$.  
 
\begin{Proposition}
\label{cosi'vuole il referee}
Let   $\Phi:\R\mapsto \R$ be a convex lowerly bounded function satisfying \eqref{e4.4} for some $p_2\geq 1$. Then for every $x\in L^{2p_2}(0,1)$ (hence, for $\mu$-a.e. $x\in H$) problem \eqref{react-diff} has a unique mild solution $X$. For every $f\in C_b(H)$ we have
\begin{equation}
\label{Kolmu}
u(x) = \int_0^{\infty} e^{-\lambda t}  \E(f(X(t,x)))\,dt, 
\end{equation}
$\mu$-a.e. $x\in H$, where $u$ is the weak solution to   \eqref{e1.1}.  
\end{Proposition}
\begin{proof}
Existence of a unique mild solution to \eqref{react-diff} follows from \cite[Theorem 5.5.8]{DPZ2}, that deals with Cauchy problems such as $dX=(AX+F(X))dt+dW(t)$, $X(0)=x$. In our case, $F(x) = -DU(x) = - \Phi'(x)$ satisfies the assumptions of \cite[Theorem 5.5.8]{DPZ2} with 
 $K= L^{2p_2}(0,1)$.  In particular, Hypothesis 5.5 is satisfied, since in \cite[Proposition 4.3]{Barcellona}  it is proved that $(t, \xi)\mapsto \int_0^t e^{(t-s)A}dW(s) (\xi)$ is a.s. continuous. 
 
  The  mild solution   is obtained as the  limit of mild solutions to approximating problems, 
$$dX_{\alpha}=(AX_{\alpha}-DU_{\alpha}(X))dt+dW(t),\quad X(0)=x ,  $$
as $\alpha \to 0$, where $DU_{\alpha}$ are the Yosida approximations of $DU$, and for each $T>0$ we have $\lim_{\alpha \to 0} \sup_{0\leq t\leq T}\|X_{\alpha}(t) - X(t)\|=0$, $\P$-a.e. By Proposition \ref{KolmLip}, 
 for every $\lambda >0$, 
\begin{equation}
\label{Kolm}
R(\lambda, K_{\alpha})f   = \int_0^{\infty} e^{-\lambda t}  \E(f(X_{\alpha}(t,\cdot))dt. 
\end{equation}
We recall that $R(\lambda, K_{\alpha})f =u_{\alpha}$ is the weak solution to \eqref{e2.14}, and that a sequence $u_{\alpha_n}$ with $\alpha_n\to 0$ converges to $u$ in $L^2(H, \mu)$  as $n\to \infty$, by Lemma \ref{l2.8}. Moreover, $\int_0^{\infty} e^{-\lambda t}  \E(f(X_{\alpha_n}(t,\cdot))dt$ goes to $\int_0^{\infty} e^{-\lambda t}  \E(f(X (t,x))dt$ pointwise $\mu$-a.e. and also in $L^2(H, \mu)$, by the dominated convergence theorem. 
Taking $\alpha= \alpha_n$ in \eqref{Kolm} and letting $n\to \infty$ formula \eqref{Kolmu} follows. 
\end{proof}

Concerning perturbed equations, 
\begin{equation}
\label{stocpert}dX=(AX-\Phi'(X)+ B(X))dt+dW(t), 
\end{equation}
we do not know about existence of invariant measures except in the case of bounded perturbations of Ornstein--Uhlenbeck equations. See \cite[Chapter  8]{DPZ2}. 
If $B$ is a bounded Borel function, Proposition \ref{p4.10} yields 
that the corresponding Kolmogorov semigroup $e^{tK_1}$ has an invariant measure $\nu$. The verification of  formula \eqref{Kolmu} where now $X(t,x)$ is the mild  solution  to \eqref{stocpert} and $u = R(\lambda, K_1)$ is not obvious. In fact, even  existence of a mild solution is not obvious. It could  be done through the Girsanov transform, but the argument is quite delicate and we hope to be able to  treat the subject in a future paper.

%%%%%%%%%%%%%%%%%%%%%%%%%%%%%%%%%%%%%%%%%%%%%%%%%
 \section{Kolmogorov equations of  stochastic Cahn--Hilliard-type problems.}
%%%%%%%%%%%%%%%%%%%%%%%%%%%%%%%%%%%%%%%%%%%%%%%%%

In Section 5 we have seen that the superposition   $x\mapsto \Phi'\circ x$ may be seen as the gradient of a suitable  function $U$ in the space $ L^2(0,1)$. This is no longer true for operators of the type  $x\mapsto \frac{d}{d\xi} (\Phi'\circ x)$ or $x\mapsto \frac{d^2}{d\xi^2} (\Phi'\circ x)$. However they   may be still interpreted as   gradients, with suitable choices of the space $H$. 

Here we set  $V:=\{ x\in H^1(0,1): \;\int_{0}^1 x(\xi)d\xi =0\}$, with scalar product $\langle x, y\rangle _V = \int_0^1 x'(\xi)y'(\xi)d\xi$, and 
we choose $H$ to be the dual space of $V$, endowed with the dual norm. 
We  consider the spaces $\widetilde{L}^p(0,1):= \{x\in L^p(0,1): \;\int_{0}^1 x(\xi)d\xi =0\}$ as  subspaces of $H$, identifying any $x\in L^p(0,1)$ with zero mean value with the element 
$y\mapsto \int_0^1 x(\xi)y(\xi)d\xi$ of $H$.

The standard extension $B$ of the negative second order derivative on $V$ with values in $H$ is defined by  
$$Bx (y) =  \int_0^1 x'(\xi)y'(\xi) d\xi, \quad y\in V.$$
If $x\in V \cap H^2(0,1)$ and $x'(0)=x'(1) =0$, then $Bx (y) =  -\int_0^1 x''(\xi)y(\xi) d\xi$ so that, with the above identification, $B$ is an extension of (minus) the second order derivative with Neumann boundary condition. The operator $B$ is an isometry between $V$ and $H$, since $\|Bx\|_{H } = \sup_{y\neq 0} 
\langle  x, y\rangle _{V} /\|y\|_{V} = \|x\|_{V}$.  Moreover, if $z\in \widetilde{L}^2(0,1)$ and $x\in V$, then $\langle z,Bx\rangle  _{H} = \langle z,x\rangle _{L^2(0,1)}$. 

Let  $e_k(\xi):=\sqrt{2 }\,\cos(k \pi\xi)$. Then $\{e_k:\;k\in \N\} $ is an orthonormal basis of $\widetilde{L}^2(0,1) $, $B e_k = k^2 \pi^2 e_k  $, and
setting $f_k = k \pi e_k$, the set  $\{f_k:\;k\in \N \}$ is  an orthonormal basis of  $H$. 
We recall that   $P_n$ is the orthogonal projection on the subspace spanned by the first $n$ elements of the basis, 
$$P_nx = \sum_{k=1}^n \langle x, f_k\rangle_{H } f_k.$$

\begin{Remark}
\label{r5.1}
{\em Note that the restriction of $P_n$ to $\widetilde{L}^2(0,1)$  is the orthogonal projection in $\widetilde{L}^2(0,1)$ on the subspace spanned by $e_1$, \ldots $e_n$. Indeed,  for every $x\in \widetilde{L}^2(0,1)$ and $k\in \N$ we have}
$$\langle x, f_k\rangle_{H } f_k = \langle  x,  B^{-1}f_k\rangle_{L^2} f_k = \langle  x,  \frac{e_k}{k \pi}\rangle_{L^2} k \pi {e_k} =  \langle  x,  e_k\rangle_{L^2}   {e_k} .$$
\end{Remark}

Here we set $A = -B^2$ and, as usual,  we denote by  $\mu$ the Gaussian measure on $H$ with zero mean and covariance $Q= -A^{-1}/2$. Note that the eigenvalues of $Q$ are now $\lambda_k := 1/2\pi^4k^4$, and $B= \sqrt{2}Q^{1/2}$.

We consider a function $\Phi:\R\mapsto \R$ satisfying the following assumptions. 

\begin{Hypothesis}
\label{Phi}
$\Phi:\R\mapsto \R$ is a $C^1$ convex lowerly bounded function, satisfying \eqref{e4.4} and  
\begin{equation}
\label{e5.1}
\lim_{r\to \pm\infty} \frac{\Phi(r)}{|r|} = +\infty .
\end{equation}
\end{Hypothesis}
Setting $p_1= p_2+1$, 
 we define $U$ as in Section 5.1, by
\begin{equation}
\label{e4.2CH}
U(x)= \left\{ \begin{array}{lll}  & \ds \int_0^{1} \Phi(x(\xi))d\xi, & x\in \widetilde{L}^{p_1}(0,1),   
\\
\\
 & +\infty , & x\notin \widetilde{L}^{p_1}(0,1). 
\end{array}\right. 
\end{equation}
$U$ is obviously convex and bounded from below, moreover by \cite[Proposition 2.8]{Barbu}, it is lower semicontinuous. To be more precise, in \cite{Barbu} the space $H$ is the dual space of $H^1_0(0,1)$, but the argument goes as well in our case. The subdifferential of $U$ is not empty at each $x\in \widetilde{L}^1(0,1)$ such that $\Phi' \circ x\in V$ and it consists of the unique element $D_0U(x) = B (\Phi' \circ x)$.

We shall see that $U\in  W^{1,2}_{1/2}(H, \mu)$, while 
$U\notin  W^{1,2}_0(H, \mu)$. For the proof,  instead of approaching $U$ by its Moreau--Yosida approximations, we shall approach it by the sequence $U\circ P_n$; namely we set 
$$U_n(x) = \int_0^1\Phi(P_nx(\xi)) d\xi , \quad  x\in H. $$
By \eqref{e4.4}, $\Phi$ satisfies \eqref{e4.1},  and we have $U(x) \leq C(1+ \|x\|_{L^{p_1}(0,1)}^{p_1})$, 
$U_n(x) \leq C(1+ \|P_nx\|_{L^{p_1}(0,1)}^{p_1})$. So, the starting point of our analysis is the study of the functions $x\mapsto  \|x\|_{L^{p}(0,1)}$, $x\mapsto  \|P_nx\|_{L^{p}(0,1)}$ for $p\geq 2$.

\begin{Proposition}
\label{p5.2}
For each  $p\geq 1$ there is $C_p>0$ such that 
\begin{equation}
\label{e5.2}
\int_H \int_0^1 |P_nx (\xi) |^{p}d\xi\,  d\mu \leq C_p\bigg(\sum_{k= 1}^n \frac{1}{k^2 \pi^2} \bigg)^{p/2}, \quad n\in \N, 
\end{equation}
\begin{equation}
\label{e5.3}
\int_H \int_0^1 |P_nx (\xi) -P_mx(\xi)|^{p}d\xi\,  d\mu \leq C_p\bigg(\sum_{k=m+ 1}^n \frac{1}{k^2 \pi^2} \bigg)^{p/2}, \quad m<n\in \N. 
\end{equation}
\end{Proposition}
\begin{proof}
First of all note that for every $x\in H$, $P_n x$ is a smooth function. Moreover for every $\xi\in (0, 1)$ and $m<n \in \N$, 
the function  $x\mapsto P_nx (\xi)- P_mx(\xi)$ is a Gaussian random variable $N_{0, \sum_{k=m+1}^n \frac{1}{\pi^4k^4}f_k(\xi)^2}$. Then, for $p\geq 1$, 
$$\int_H |P_nx (\xi) - P_mx(\xi)|^{p} d\mu = \int_{\R}|\eta| ^{p}N_{0, \sum_{k=m+1}^n \frac{1}{\pi^4k^4}  f_k(\xi)^2}(d\eta) $$
$$= c_p 
 \bigg(\sum_{k=m+1}^n \frac{1}{k^2 \pi^2} e_k(\xi)^2 \bigg)^{p/2} \leq 2^{p/2}c_p  \bigg(\sum_{k=m+1}^n \frac{1}{k^2 \pi^2} \bigg)^{p/2}, $$
 so that 
$$\int_H \int_0^1 |P_nx (\xi) - P_mx(\xi)|^{p} d\xi\, d\mu = \int_0^1 \int_H |P_nx (\xi) - P_mx(\xi)|^{p}  d\mu \, d\xi \leq 2^{p/2}c_p  \bigg(\sum_{k=m+1}^n \frac{1}{k^2 \pi^2} \bigg)^{p/2};$$
that is, \eqref{e5.3} holds. The proof of \eqref{e5.2} is the same. 
\end{proof}

Proposition \ref{p5.2} has several consequences. 

\begin{Corollary}
\label{c5.3}
$\mu(\widetilde{L}^p(0,1) )= 1$, and the sequence of functions $(x,\xi) \mapsto P_nx(\xi)$ converges to $(x,\xi)\mapsto x(\xi)$ in $L^p(H\times (0,1), \mu\times d\xi)$, for every $p\geq 1$. 
\end{Corollary} 
\begin{proof}
It is sufficient to prove that  the statement holds for $p=2$. Indeed, estimate \eqref{e5.3} implies that the sequence $(x, \xi)\mapsto P_nx (\xi)$ converges in 
$L^p(H\times (0,1), \mu\times d\xi)$ for every $p$ to a limit function, that we identify with the function $(x,\xi) \mapsto x(\xi)$ taking $p=2$. Once we know that $\int_H\int_0^1 |x(\xi)|^p d\xi \,d\mu <\infty$, then 
$\mu(\widetilde{L}^p(0,1)) $ is obviously $1$. 

So, fix $p=2$. Since 
$$\int_0^1 |P_nx(\xi)|^2 d\xi = 
\int_0^1 \sum_{h, k=1}^n \langle x, f_k\rangle_{H } \langle x, f_h\rangle_{H } f_k(\xi)f_h(\xi)d\xi = \int_0^1 \sum_{k=1}^n \langle x, f_k\rangle^2_{H } f_k(\xi)^2d\xi , $$
then for every $x\in H$ the sequence $\int_0^1 |P_nx(\xi)|^2 d\xi $ is increasing,  it converges to 
$\|x\|^2_{L^2}$ if $x\in \widetilde{L}^2(0,1)$, and to $+\infty$ if $x \notin \widetilde{L}^2(0,1)$ by Remark \ref{r5.1}. 
By monotone convergence and \eqref{e5.2} with $p=2$ the limit function belongs to $L^1(H, \mu)$, 
and this implies $\mu(\widetilde{L}^2(0,1) ) = 1$. Consequently, the  function $(x, \xi)\mapsto x(\xi)$ is defined a.e. in  $H\times (0, 1)$. 
Moreover, 
$$\int_{\widetilde{L}^2(0,1)} \int_0^1|P_nx(\xi) - x(\xi)|^2d\xi\,d\mu = \int_{\widetilde{L}^2(0,1)} \lim_{m\to \infty} \int_0^1|P_nx(\xi) - P_m x(\xi)|^2d\xi\,d\mu $$
$$\leq \liminf_{m\to \infty}\int_{\widetilde{L}^2(0,1)}   \int_0^1|P_nx(\xi) - P_m x(\xi)|^2d\xi\,d\mu $$
For each $\eps>0$ there is $n_{\eps}\in \N$ such that for  $n$, $m\geq n_{\eps} $ we have 
$\int_{\widetilde{L}^2(0,1)}   \int_0^1|P_nx(\xi) - P_m x(\xi)|^2d\xi\,d\mu \leq \eps$. Then for $n\geq n_{\eps}$
we get  $\int_{\widetilde{L}^2(0,1)} \int_0^1|P_nx(\xi) - x(\xi)|^2d\xi\,d\mu \leq \eps$, and the statement follows. 
\end{proof}

\begin{Proposition}
\label{p5.4}
Under Hypothesis \ref{Phi}, $U \in W^{1,p}_{1/2}(H, \mu)$ and $ \lim_{n\to \infty} U_n = U$ in $L^p(H, \mu)$,  for every $p\geq 1$. 
Moreover, $D_kU(x) = \int_0^1\Phi '(x(\xi))f_k(\xi)d\xi$ for a.e. $x\in H$. 
\end{Proposition} 
\begin{proof} 
As a first step, we remark that the sequence of functions $x\mapsto  \|P_nx\|_{L^p(0,1)}^p$ is bounded in $L^s(H, \mu)$ for every $s\geq 1$. Indeed, using  the H\"older inequality we get 
$$ \int_0^1 |P_nx(\xi)|^pd\xi   \leq \bigg(\int_0^1 |P_nx(\xi)|^{ps}d\xi \bigg)^{1/s}, \quad s\geq 1,  $$
and the right-hand  side belongs to $L^{s}(H, \mu)$ with norm independent of $n$, by estimate \eqref{e5.2}. 

We already remarked that  $|U_n(x)|\leq \int_0^1 C(1+|P_nx(\xi)|)^{p_1}d\xi$ with $p_1=p_2+1$, so that 
 $U_n$ is bounded in $L^p(H, \mu)$ by a constant independent of $n$, for every $p\geq 1$. 
Let us prove that $U_n\to U$ in $L^p(H, \mu)$. Using \eqref{e4.4} and the H\"older inequality we get
$$\begin{array}{l} 
\ds |U_n(x) -U(x)|^p \leq \bigg(\int_0^1 |\Phi(P_nx(\xi)) - \Phi (x(\xi))|d\xi\bigg)^p
\\
\\
\ds \leq C^p \bigg(\int_0^1(1+ |x(\xi)| + |P_nx(\xi)|)^{p_2}|P_nx(\xi) - x(\xi)| d\xi \bigg)^p
\\
\\
\leq \ds C^p \bigg(\int_0^1(1+ |x(\xi)| + |P_nx(\xi)|)^{2p_2p}d\xi\bigg)^{1/2} 
\bigg(\int_0^1|P_nx(\xi) - x(\xi)|^{2p} d\xi \bigg)^{1/2} .
\end{array}$$
Since $x\mapsto \|1 +|x| + |P_nx|\,\|_{L^{2p_2p}(0,1)}$ is bounded in $L^{2p_2p}(H, \mu)$ by a constant independent of $n$, and $\|P_nx -x\|_{L^{2p}(0,1)}$ vanishes in $L^{2p }(H, \mu)$ as $n\to \infty$, 
by the H\"older inequality the right-hand  side vanishes in $L^{1}(H, \mu)$ as $n\to \infty$.
Hence, $U$ in $L^{p }(H, \mu)$ and 
$U_n\to U$ in $L^{p }(H, \mu)$ as $n\to \infty$. 

To prove that $U\in W^{1,p}_{1/2}(H, \mu)$ it is enough to show that the sequence $U_n$ is bounded in 
$W^{1,p}_{1/2}(H, \mu)$ (e.g., \cite[Lemma 5.4.4]{Boga}). We already know that it is bounded in $L^{p}(H, \mu)$. Moreover each  $U_n$ is continuously differentiable, since it is the composition of $x\mapsto P_nx$ which is smooth from $H$ to $C([0,1])$, and $y\mapsto \int_0^1\Phi(y(\xi))d\xi$ which is continuously differentiable  from $C([0,1])$ to $\R$,  and 
\begin{equation}
\label{e5.4}
D_k U_n(x ) =  \int_0^1 \Phi'(P_nx(\xi)) f_k(\xi)d\xi , \quad k\leq n, 
\end{equation}
while $D_k U_n(x )=0$ for $k>n$. 
Using again assumption \eqref{e4.4} and the H\"older inequality we get
$$|D_k U_n(x)| =  \bigg|\int_0^1 \Phi'(P_nx(\xi)) f_k(\xi)d\xi \bigg|
\leq C \int_0^1 (1+ |P_nx(\xi)|)^{p_2}|f_k(\xi)|d\xi 
\leq \frac{C}{\lambda_k^{1/4}} \|1+ |P_nx|\,\|_{L^{2p_2}(0,1)}^{p_2},  $$
for $k\leq n$. Then 
$$\|Q^{1/2} DU_n(x)\|^2 = \sum_{k=1}^{n}  \lambda_k |D_kU_n(x)|^2\leq 
 C^2   \sum_{k=1}^{\infty}  \lambda_k^{1/2}  \|1+|P_nx|\,\|_{L^{2p_2}(0,1)}^{2p_2}. $$
By the first part of the proof we know that $x\mapsto \|P_nx\|_{L^{2p_2}(0,1)}^{2p_2}$
belongs to   $L^1(H, \mu)$ with norm bounded by a constant independent of $n$. Since  $\sum_{k=1}^{\infty}  \lambda_k^{1/2}<\infty$,  then $U_n$ is bounded in 
$W^{1,p}(H, \mu)$ so that $U\in W^{1,p}(H, \mu)$. 

Now we show that for every  $k\in \N$,  a subsequence of $D_kU_n$ converges to $\int_0^1\Phi ' (x(\xi))f_k(\xi)d\xi$ in $L^2(H, \mu)$. Then the equality  $D_kU(x) = \int_0^1\Phi ' (x(\xi))f_k(\xi)d\xi$
$\mu$-a.e.  follows using the  integration by parts formula \eqref{e1.7}. 

We have
$$\int_H \bigg| D_kU_n(x) - \int_0^1\Phi ' (x(\xi))f_k(\xi)d\xi\bigg|^2 d\mu 
\leq \int_H \int_0^1 |\Phi ' (P_nx(\xi)) - \Phi ' (x(\xi))|^2f_k(\xi)^2d\xi d\mu .$$
By Corollary \ref{c5.3}, the sequence of functions $(x, \xi)\mapsto P_nx(\xi)$ converges to $x(\xi)$ in $L^2(H, \mu)$. Consequently, a subsequence converges $\mu$-almost everywhere, and since $\Phi '$ is continuous, along such subsequence $(x, \xi)\mapsto (\Phi '(P_nx(\xi)) - \Phi '(x(\xi)))f_k(\xi)$ vanishes.
Moreover, by assumption \eqref{e4.4}, 
$$|\Phi ' (P_nx(\xi)) - \Phi ' (x(\xi))|^2f_k(\xi)^2 \leq 
C^2(2+ |P_nx(\xi)|^{p_2} + |x(\xi)|^{p_2})\|f_k\|_{\infty}^2$$
which belongs to $L^1(H\times (0,1), \mu \times d\xi )$ with norm bounded by a constant independent of $n$. The statement follows by the dominated convergence theorem. 
\end{proof}

Then, $U$ satisfies Hypothesis \ref{Hyp}. So, the results of Theorem \ref{t2.7} and of Proposition \ref{p3.2}, \ref{p4.8} hold.

\vspace{2mm}
 We recall that the operator $Q^{1/2}D$ in the space $L^2(H, \nu;H)$
is the closure of the operator $\varphi \mapsto Q^{1/2}D\varphi$ defined in a set of smooth functions, see Definition \ref{d1.4}. However, we can identify $Q^{1/2}DU(x)$: indeed, recalling that $B= Q^{-1/2}/\sqrt{2}$, we obtain 
$$D_kU(x) = \langle \Phi '\circ x, f_k\rangle_{L^2(0,1)} = \langle \Phi '\circ x -\int_0^1\Phi'(x(\xi))d\xi, Bf_k\rangle_{H } =  \frac{\lambda_k^{-1/2}}{\sqrt{2}}\langle \Phi '\circ x  -\int_0^1\Phi'(x(\xi))d\xi,  f_k\rangle_{H } $$
for every $x\in \widetilde{L}^{2p_2}(0,1)$, so that 
$$Q^{1/2}DU(x) = \frac{1}{\sqrt{2}}\sum_{k=1}^{\infty}   \langle  \Phi '\circ x  -\int_0^1\Phi'(x(\xi))d\xi,  f_k\rangle_{H }f_k =  \frac{\Phi '\circ x  -\int_0^1\Phi'(x(\xi))d\xi}{\sqrt{2}}.$$
On the other hand, we
already mentioned that if  $\Phi '\circ x\in V$ [i.e., $\Phi '\circ x\in D(B)$], then $D_0U(x) = B( \Phi '\circ x)$, so that, since $Q^{1/2}= B^{-1}/\sqrt{2}$,  $Q^{1/2}D_0U(x) = Q^{1/2}DU(x)$. 
For such $x$ we have
$$ \langle B(\Phi '\circ x) ,  Du(x)\rangle  = \langle {\Phi '\circ x} , BDu(x)\rangle =
\langle Q^{1/2}DU(x), Q^{-1/2}Du(x)\rangle   
=\langle  DU(x), Du(x)\rangle .$$
Then  the stochastic differential equation \eqref{e1.3}
in $H$  reads as 
\begin{equation}
\label{CHstoch}
dX(t)= (- \tfrac{\partial ^4}{\partial \xi^4} X -  \tfrac{\partial ^2}{\partial \xi^2}\Phi'(X))dt+dW(t),\quad X(0)=x,  
\end{equation}
and  it is a stochastic Cahn--Hilliard equation, whose Kolmogorov operator is   $\mathcal K$. It was studied in 
\cite{EM} and in several following papers, in particular in \cite{DPD} where existence and uniqueness of weak solutions were proved for polynomial nonlinearities $\Phi$. Here $W(t)$ is, as usual, any $H$-valued cylindrical Wiener process defined in a probability space $(\Omega, \mathcal F, \P)$. 

We think that  it is possible to relate the weak solution to \eqref{CHstoch} constructed in  \cite{DPD} to the solution of the Kolmogorov equation by formula \eqref{finalmente}, at least in the model case $\Phi(\xi) = \xi^{2m}$ with $m\in \N$. 
Indeed, for every $x\in H$ the weak solution given by  \cite[Theorem 2.1]{DPD} is obtained through cylindrical approximations $X_n(t)$, solutions to
\begin{equation}
\label{Kolmuzn}
dX_{n}=(A_nX_{n} + P_n B\Phi'(P_nX))dt+P_ndW(t),\quad X_n(0)=P_nx ,  
\end{equation}
with $A_n = A_{|P_n(H)} \in {\mathcal L}(P_n(H))$; identifying $P_n(H)$ with $\R^n$ 
the  Kolmogorov operator ${\mathcal K}_n$ associated to \eqref{Kolmuzn} is 
$${\mathcal K}_n\varphi = \frac{1}{2} \Delta \varphi - \sum_{k=1}^n \bigg(\frac{x_k}{2\lambda_k} + \int_0^1 \Phi'\bigg( \sum_{h=1}^n x_h f_h(\xi)\bigg) f_k(\xi)d\xi\bigg)D_k\varphi  .$$
Taking into account such explicit expressions, one should be able to follow the procedure of Proposition \ref{KolmLip} (that deals with the case of Lipschitz continuous $DU$). However, many details should be fixed, and giving a complete proof goes beyond the aims of this paper.

\section{Acknowledgements}
We thank one of the Referees for careful reading of the manuscript, and for several remarks that helped to improve the paper.

\end{document}